\documentclass{article}

\usepackage{amsmath,amssymb,amsthm} 
\numberwithin{equation}{section}
\usepackage{graphics}  
\usepackage{bbm}
\usepackage{makeidx}
\makeindex
\usepackage{mathabx}
\usepackage{ulem}

\newcommand\F{\mathcal F}

\renewcommand{\P}{\mathbb{P}}

\newcommand{\im}{\textup{i}}

\newcommand\eps{\varepsilon}
\newcommand\wtilde{\widetilde}

\newcommand\p{\mathbb{P}}

\newcommand\E{\mathbb{E}}

\newcommand{\cp}{\overset{\mathbb{P}}{\longrightarrow}}
\newcommand{\st}{\overset{d_{st}}{\longrightarrow}}
\newcommand{\cd}{\overset{d}{\longrightarrow}}    

\newtheorem{theorem}{Theorem}[section]

\newtheorem{corollary}[theorem]{Corollary}

\newtheorem{lemma}[theorem]{Lemma}
\newtheorem{proposition}[theorem]{Proposition}

\theoremstyle{definition}
\newtheorem{example}[theorem]{Example}
\newtheorem{remark}[theorem]{Remark}

\usepackage{color}
\usepackage{comment}
\newif\ifcol
\coltrue
\ifcol

\newcommand{\colorr}{\color[rgb]{0.8,0,0}}

\newcommand{\coloroy}{\color[rgb]{1,0.95,0}}

\else
\newcommand{\colorr}{\color{black}}% {{\color[rgb]{0.8,0,0}}
\newcommand{\coloroy}{\color{black}}% {\color[rgb]{1,0.95,0}}
\fi
\excludecomment{en-text}
\includecomment{jp-text}
\includecomment{comment}
\def\tti{{\tt i}}
\def\tik{{t_{ik_n}}}
\def\tlk{{t_{lk_n}}}
\def\tlkp{{t_{(l+1)k_n}}}
\def\sfh{{\sf h}}

\def\bd{\begin{description}}
\def\ed{\end{description}}
\def\koko{{\coloroy{koko}}}

\def\halflineskip{\vspace*{3mm}}
\def\nn{\nonumber}
\def\be{\begin{equation}}
\def\ee{\end{equation}}
\def\bea{\begin{eqnarray}}
\def\eea{\end{eqnarray}}
\def\beas{\begin{eqnarray*}}
\def\eeas{\end{eqnarray*}}
\def\bi{\begin{itemize}}
\def\ei{\end{itemize}}
\def\bd{\begin{description}}
\def\ed{\end{description}}
\def\ep{\varepsilon}
\def\half{\frac{1}{2}}

\def\calc{{\cal C}}

\def\cale{{\cal E}}
\def\calf{{\cal F}}

\def\call{{\cal L}}

\newcommand{\bbD}{{\mathbb D}}
\newcommand{\bbE}{{\mathbb E}}
\newcommand{\bbF}{{\mathbb F}}
\newcommand{\bbG}{{\mathbb G}}
\newcommand{\bbH}{{\mathbb H}}

\newcommand{\bbN}{{\mathbb N}}

\newcommand{\bbP}{{\mathbb P}}

\newcommand{\bbR}{{\mathbb R}}

\newcommand{\bbZ}{{\mathbb Z}}

\begin{document}
\title{Edgeworth expansion for the pre-averaging estimator\footnote{
Mark Podolskij and Bezirgen Veliyev gratefully acknowledge financial support 
from CREATES funded by the Danish National Research Foundation (DNRF78). 
The research of Nakahiro Yoshida is supported by 
Japan Society for the Promotion of Science Grants-in-Aid for Scientific Research Nos. 24340015 
(Scientific Research), 
Nos. 24650148 and 26540011 (Challenging Exploratory Research); 
CREST Japan Science and Technology Agency; 
NS Solutions Corporation; 
and by a Cooperative Research Program of the Institute of Statistical Mathematics.
}
}
\author{Mark Podolskij\thanks{Department of Mathematics, Aarhus University,  Ny Munkegade 118, 8000 Aarhus C,
Denmark, Email: mpodolskij@math.au.dk} 
\and Bezirgen Veliyev\thanks{CREATES, Department of Economics and Business Economics, Aarhus University,  Fuglesangs Alle 4, 8210 Aarhus V,
Denmark, Email: bveliyev@econ.au.dk} \thanks{corresponding author}
\and Nakahiro Yoshida\thanks{Graduate School of Mathematical Science, 
The University of Tokyo, 
3-8-1 Komaba, Meguro-ku, Tokyo 153, Japan, Email: nakahiro@ms.u-tokyo.ac.jp} \thanks{CREST Japan Science and Technology Agency}
}

\maketitle

\begin{abstract}
In this paper, we study the Edgeworth expansion for a pre-averaging estimator of quadratic variation
in the framework of continuous diffusion models observed with noise. More specifically, we obtain 
a second order expansion for the joint density of the estimators of quadratic variation and its asymptotic
variance. Our approach is based on martingale embedding, Malliavin calculus and stable central limit theorems
for continuous diffusions. Moreover, we derive the density expansion for the studentized statistic,
which might be applied to construct asymptotic confidence regions.

\ \

\bigskip


\end{abstract}

\section{Introduction}
In the last decade the estimation of quadratic variation of It\^o semimartingales have been investigated 
by many researchers. Typically, this estimation problem is considered in the infill asymptotics setting, i.e.
the underlying observations are recorded from high frequency data of continuous/discontinuous 
It\^o semimartingales, diffusion processes corrupted by noise or related models. A recent comprehensive monograph 
\cite{JP12}
presents a detailed asymptotic analysis for estimators of quadratic variation and related objects in various frameworks.

In financial mathematics, it is nowadays widely accepted that financial data is contaminated by 
{\it microstructure noise} such as rounding errors, bid-ask bounds and misprints, when observed at ultra high frequency. 
This fact prevents us from using classical realised volatility estimator at such frequencies. In this work
we consider a continuous SDE model corrupted by additive i.i.d. noise, i.e. the observations are
\[
Y_{t_i} = X_{t_i} + \eps_{t_i},
\] 
where $X$ is a continuous diffusion process, $\eps$ is an i.i.d process independent of $X$ and $t_i=i\Delta_n$
with $\Delta_n\rightarrow 0$. It is well-known that realised volatility has an explosive behaviour and more 
delicate methods are required to estimate the quadratic variation of the latent diffusion process $X$. The most
famous estimation approaches in this framework are the {\it multiscale approach} of \cite{Z06},
the {\it realised kernel method} proposed in \cite{BHLS08} and the {\it pre-averaging concept} originally introduced in \cite{POVE09}
among others. All these estimators are consistent, asymptotically mixed normal and have the convergence rate $\Delta_n^{-1/4}$,
which is known to be optimal. 

Due to this relatively slow rate of convergence the quality of the mixed normal approximation is rather questionable 
even at high frequencies. The aim of this paper is to derive the second order Edgeworth expansion for the pre-averaging 
estimator to improve the mixed normal approximation of the unknown density.
We remark that our work is related to a recent paper \cite{POYO13}, which investigates the Edgeworth expansion
for power variations of continuous diffusion processes in the noise-free setting. 
However, in the framework of continuous diffusions corrupted by additive i.i.d. noise the stochastic 
second order expansion of the pre-averaging statistics is more involved. 
Our methodology relies on martingale 
methods, stochastic expansion of the pre-averaging statistics and general theory of Edgeworth expansion associated
with mixed normal limits studied in \cite{YOSH13}. The latter approach is heavily using different aspects of Maliavin calculus,
such as integration by parts formula and conditions for smoothness of probability laws. 
In a second step, we will present the Edgeworth expansion for the density of the studentized statistic, which can be
potentially  used
to construct more precise confidence regions for the quadratic variation of the diffusion process $X$.

The paper is organised as follows. We describe the main setting and recall the pre-averaging approach in Section
\ref{sec:setting}. Section \ref{sec3} presents a second order stochastic decomposition for the pre-averaging estimator
of the quadratic variation. We demonstrate the general theory of Edgeworth expansion with respect to mixed normal 
limits in Section \ref{sec4}. In Section \ref{sec5} we apply the asymptotic theory to the pre-averaging estimator
and present Edgeworth expansion for the studentized version of the statistic. 
In Section 6, we deal separately  with the case of constant volatility, which does not satisfy our non-degeneracy condition. 
Section \ref{sec6} demonstrates an example 
and Section \ref{sec7} collects some steps of the proof.

\section{Setting} \label{sec:setting}
In this paper, we deal with infill asymptotics, i.e. the data is observed 
at equidistant grid  $i \Delta_n,$ $i \in \mathbb{N},$ over a finite horizon $[0,T]$ and 
$\Delta_n \to 0$. We also impose that $1/\Delta_n$ is a positive integer. The terminal time $T$ is assumed to be fixed and we assume $T=1$ without loss of generality.
For simplicity, we use the notation $t_{i}:=i \Delta_n.$
On a filtered probability space $(\Omega, \mathcal{F}, (\mathcal{F}_t)_{t \in [0,1]}, \mathbb{P})$ (to be specified in Section \ref{SLT}),  we consider a diffusion model that satisfies the stochastic differential equation 
\begin{align} \label{SDE}
dX_t=b^{[1]}(X_t) dW_t+b^{[2]}(X_t) dt
\end{align}
with a bounded random variable $X_0$ as starting value,  a standard Brownian motion $W$ and 
continuous functions $b^{[1]}, b^{[2]}:\mathbb{R} \to \mathbb{R}.$ 
We intentionally choose the  notations $b^{[1]}, b^{[2]}$ to emphasize that the diffusion term 
$b^{[1]}$ dominates the drift term $b^{[2]}$ in asymptotic expansions throughout the paper.
We are interested in estimating the integrated volatility which we denote by
\begin{equation}
V=\int_0^1 (b^{[1]}(X_t))^2 dt.
\end{equation}
However, due to the microstructure noise effects, we are not able to observe the process $X$ directly, but only with distortions. 
More specifically, we assume that the underlying observations $(Y_{t_i})_{i\geq 0}$ are given by 
\begin{equation}
Y_{t_i}=X_{t_i}+ \eps_{t_i},
\end{equation}
where $(\eps_{t_i})_{i\geq  0}$ is a sequence of  i.i.d. random variables with
\begin{equation}
\E[\eps_{t_i}]=0 \qquad \text{and} \qquad \E[\eps_{t_i}^2]=\omega^2>0, 
\end{equation}
and $\eps_{t_i}$ is $\calf_{t_i}$-measurable.
In addition, we assume that the processes $\eps$ and $X$ are independent. Such additive noise models are widely
used in financial mathematics, see e.g. \cite{BHLS08,JLMP09,POVE09,ZHMA05} among many others.

We require some notation to describe the pre-averaging estimator which is originally due to \cite{JLMP09,POVE09}.
We pick a sequence of positive integers $(k_n)_{n=1}^{\infty}$ and a positive real number $\theta$  such that 
\begin{equation}\label{kn}
k_n\Delta_n^{1/2}=\theta +o(\Delta_n^{1/2})   \mbox{ and } 
d_n:=\bigg\lfloor\frac{1}{k_n \Delta_n}\bigg\rfloor \in \mathbb{N}.
\end{equation}
Moreover, we consider a continuous, non-negative function $g$ on $[0,1]$ which is piecewise continuously differentiable with a piecewise Lipschitz derivative $g'.$ This function should also satisfy 
\begin{equation}
g(0)=g(1)=0 \quad
\mbox{ and } \int_0^1 g^2(s)ds>0.
\end{equation}
Furthermore, we introduce the following notations associated with $g$:
\begin{align*}
&h(j/k_n)=g((j+1)/k_n)-g(j/k_n), ~ \psi_1^n=k_n \sum_{j=0}^{k_n-1} (h(j/k_n))^2, \\
 &\psi_2^n=\frac{1}{k_n} \sum_{j=1}^{k_n-1} (g(j/k_n))^2, ~ \psi_3^n=\frac{1}{k_n} \sum_{j=1}^{k_n-1} g(j/k_n), \\
 &\psi_4^n=\frac{1}{(k_n)^2} \sum_{j=1}^{k_n-1} (g(j/k_n))^2 (j-1/2).
\end{align*}
Moreover, we need the notations written below. The first four are limits of the terms $\psi_i^n, 1 \leq i \leq 4.$
\begin{align*}
\psi_1&=\int_0^1  [g'(s)]^2 ds, ~ \psi_2=\int_0^1 g^2(s) ds, ~ \psi_3=\int_0^1 g(s) ds, ~ \psi_4=\int_0^1  s g^2(s) ds, \\
\psi_5&=\int_{0}^{1} \int_{0}^{u} g(s) ds g^2(u)du, ~ \psi_6=\int_{0}^{1} \left( \int_{0}^{u} g(s) ds g(u)-\int_0^u (g(s)-\psi_2) ds
 \right)^2 du \\
 \psi_7&=\int_{0}^{1} \int_{0}^{u} \int_0^s [2 g(s)+g(r)]^2 dr ds g^2(u) du.
\end{align*}
For any process $U$ we define the pre-averaged increment at stage $i \Delta_n$ via 
\begin{align*}
\widebar {U}_{t_{i}}=\sum_{j=1}^{k_n-1} g(j/k_n)\Delta_{i+j}^n U=\sum_{j=0}^{k_n-1} -h(j/k_n) U_{{t_{i+j}}},
\end{align*}
where $\Delta_i^n U =U_{t_{i}}-U_{t_{i-1}}$. 
Finally, we are ready to introduce the pre-averaging  estimator for the quadratic variation $V:$
\begin{equation} \label{PAEstimator}
V_n=\frac{1}{\psi_2^n} \sum_{i=0}^{d_n-1} (\widebar Y_{t_{i k_n}})^2- \frac{\psi_1^n d_n \Delta_n  }{2 \psi_2^n k_n} \sum_{i=1}^{ 1/\Delta_n } (\Delta_i^n Y)^2.
\end{equation}
We remark that $V_n$ is essentially the estimator proposed in \cite{JLMP09} with the difference that we only use
non-overlapping windows in this paper. This makes it easier to determine the dominating martingale 
$M_n$ of the estimator, which is required in Sections \ref{sec3} and \ref{sec4}, while computation of the martingale part of the original estimator investigated 
in \cite{JLMP09} is far from being obvious.
As we will see below, we need a consistent estimator of the asymptotic conditional 
variance associated with $V_n$, which is defined as
\begin{equation}\label{fn}
F_n=\frac{2 \Delta_n^{-1/2}}{3 (\psi_2^n)^2} \sum_{i=0}^{d_n-1} (\widebar Y_{t_{i k_n}})^4.
\end{equation}

We recall that a sequence of random variables $(Y_n)_{n \geq 1},$ which are defined on $(\Omega, \mathcal{F}, \mathbb{P})$
and take values in a metric space $E,$ is said to converge stably with the limit $Y,$ which is defined on an extension $(\overline{\Omega}, \overline{\mathcal{F}}, \overline{\p})$ of $(\Omega, \mathcal{F}, \p),$ if for any bounded, continuous function $f$ and any bounded
$\mathcal{F}$-measurable random variable $Z$ it holds that
\begin{equation}\label{eq:defstable}
\E[f(Y_n) Z] \to \overline \E[f(Y) Z], ~~~ n \to \infty.
\end{equation}
In short, we use the notation $Y_n \st Y.$ We say that $Y$ is mixed normal with random variance $Z^2$, and write $Y \sim MN(0,Z^2),$ 
if $Y \sim Z U, $ where $U \sim N(0,1),$ $Z>0$ and $U,$ $Z$ are independent. 

Denoting $Z_n=\Delta_n^{-1/4}(V_n-V),$ we proceed to the first asymptotic result whose proof essentially follows from the work 
of \cite{JLMP09,POVE09}. We provide a sketch of the proof in Section \ref{sec7}.
%\begin{theorem} \label{LLN}
%Suppose that $\E[(\eps_{t_1})^4]<\infty.$ Then, we get
%$$V_n \cp V.$$
%\end{theorem}

\begin{theorem} \label{CLT}
Under the condition $\E[(\eps_{t_1})^8]<\infty,$ we deduce that
$$Z_n \st M \sim MN(0,C) \mbox{ with } C=2 \theta \int_0^1  \left((b^{[1]}(X_t))^2+ \frac{\omega^2 \psi_1}{\theta^2 \psi_2} \right)^2 dt.
$$
Moreover, we obtain $$F_n \cp C.$$
\end{theorem}
We note that due to $$Z_n \st M \mbox{ and } F_n \cp C,$$ 
and the properties of stable convergence, the studentized statistic satisfies $$\frac{Z_n}{\sqrt{F_n}} \cd N(0,1).$$
In this paper, we will first derive an asymptotic expansion for the pair $(Z_n, F_n)$ and then proceed to calculate the related Edgeworth expansion for the studentized statistic $$\frac{Z_n}{\sqrt{F_n}}.$$
\begin{example}
Our prime example for $g$ is the function $$g(x)=x\wedge (1-x).$$
In this case, we obtain
\begin{align*}
\psi_1^n=1,~
\psi_2^n=\frac{k_n^2+2}{12 k_n^2},~ \psi_3^n=\frac{1}{4}, ~
\psi_4^n=\frac{k_n^2+2}{24 k_n^2} \mbox{  when } k_n \mbox{ is even}
\end{align*}
and
\begin{align*}
\psi_1^n=\frac{k_n-1}{k_n},~
\psi_2^n=\frac{k_n^2-1}{12 k_n^2},~ \psi_3^n=\frac{k_n^2-1}{4 k_n^2}, ~
 \psi_4^n=\frac{k_n^2-2}{24 k_n^2} \mbox{  when } k_n \mbox{ is odd}.
 \end{align*}
In addition, we get
\begin{align*}
\psi_1=1, ~ \psi_2=\frac{1}{12}, ~ \psi_3=\frac{1}{4}, ~ \psi_4=\frac{1}{24}, ~ \psi_5=\frac{1}{96}, ~ \psi_6=\frac{143}{24192}, ~ \psi_7=\frac{1}{105}.
\end{align*}
\end{example}

\section{Stochastic decomposition of $Z_n$} \label{sec3}
In this section, we provide a stochastic decomposition for the bias corrected version of the random variable $Z_n$ defined in the previous section.
This second order stochastic expansion is essential for obtaining the Edgeworth expansion discussed in Sections \ref{sec4} and
\ref{sec5}. Since the first term of the estimator $V_n$ defined in \eqref{PAEstimator} uses the observations $Y_{t_i}$ with $0\leq i \leq d_n k_n$, we effectively estimate the quadratic variation of 
$X$ over the interval $[0, d_n k_n \Delta_n]$. For this reason, we consider the bias corrected statistic
\begin{align} \label{znbar}
\overline{Z}_n= Z_n + \Delta_n^{-1/4}\int_{d_n k_n \Delta_n}^1 (b^{[1]} (X_t))^2 dt.
\end{align}  
Obviously, the statistic $\overline{Z}_n$ also satisfies Theorem \ref{CLT}, since the correction converges to $0$ in probability. However, the bias may affect the higher order asymptotics. 
In the next step, we proceed with the estimation of the bias to construct a feasible statistic. We basically follow the procedure proposed in \cite[Section 4]{HP13}.  Let $p_n$ be a sequence of integers satisfying $p_n \to \infty$, $p_n \Delta_n \to 0$ and  $p_n \sqrt{\Delta_n} 
\to \infty$, and set $J_n:=\{1/ \Delta_n-p_n+1, \ldots,   1/ \Delta_n\}$. For each $
t \in [d_n k_n \Delta_n, 1]$ we define
\begin{align*}
(\widehat{b}^{[1]} (X_t))^2 := \frac{1}{\psi_2^n k_n \Delta_n p_n} \sum_{i+k_n \in J_n} 
(\widebar {Y}_{t_{i}})^2  - \frac{\psi_1^n}{2 \theta^2 \psi_2^n p_n}\sum_{ i \in J_n} (\Delta_i^n Y)^2,
\end{align*}
which is constant in $t$. It has been proved in \cite{HP13} that this local estimator is consistent for 
$(b^{[1]} (X_t))^2$. Thus, a feasible version of the statistic $\overline{Z}_n$ is obtained via
\begin{align} \label{znstar}
Z^{\star}_n= Z_n + \Delta_n^{-1/4}\int_{d_n k_n \Delta_n}^1 (\widehat{b}^{[1]} (X_t))^2 dt.
\end{align}
We remark that $1- d_n k_n \Delta_n = O(\Delta_n^{1/2})$, which implies that 
\begin{align} \label{znnegl}
Z_n^{\star} - \overline{Z}_n = o_{\mathbb{P}} (\Delta_n^{1/4}). 
\end{align}
In the next step,  we will  show that
$$Z_n^{\star}=M_n+\Delta_n^{1/4} N_n
$$
where $M_n$ and $N_n$ are some tight sequences of random variables. Before we go into details, we need more notations.
We again consider the SDE defined in \eqref{SDE}. However, in this section we assume that $b^{[1]},  b^{[2]} \in C^4(\mathbb{R}).$
Under this smoothness assumptions, we apply Ito's lemma and  write $b^{[k]}(X_t),k=1,2$ in the stochastic differential form as
$$d b^{[k]}(X_t)=b^{[k.1]}(X_t) dW_t+b^{[k.2]}(X_t) dt.$$
Similarly, we define the processes $b^{[k_1.k_2.k_3]}(X_t),$ $k_1,k_2,k_3=1,2.$
Throughout the paper, we will also use the shorthand notations 
\begin{equation}
b_t^{[k_1, \ldots, k_d]}=b^{[k_1, \ldots, k_d]}(X_t),~ d=1,2,3,~ k_1, \ldots, k_d=1,2.
\end{equation}
The following process, which is the first order approximation 
of $\widebar Y_{t_{i k_n}},$ will play an important role throughout the proofs:
\begin{equation}
\alpha_{t_{i k_n}}=b^{[1]}_{t _{i k_n}} \widebar{W} _{t _{i k_n}}+\widebar \eps _{t _{i k_n}}.
\end{equation}
Note that the quantity $\alpha_{t_{i k_n}}$ is obtained from $\widebar Y_{t_{i k_n}}$ via freezing the volatility
process at time $t_{i k_n}$ and ignoring the drift process $b^{[2]}$. 
We also need to define a function $g_n$ and a process  $W(i, t)$:
\begin{align} \label{processW(i,t)}
g_n(s)&=\sum_{j=1}^{k_n-1} g(j/k_n) 1_{((j-1)\Delta_n, j\Delta_n]}(s), 
\\
W(i, t)&=\int_{t_{i k_n}\wedge t}^{t_{(i+1)k_n}\wedge t} g_n(u-t_{i k_n}) dW_u.
\end{align}
We note that $g_n(s)$ vanishes for $s \leq 0$ and $s>(k_n-1)\Delta_n.$
Moreover, we obtain the identity $$\widebar W_{t_{i k_n}}=W(i, t_{(i+1) k_n}).$$ 
The next proposition, which gives the expansion of $Z_n^{\star}$, 
is a central result of this section. 

\begin{proposition}\label{StochExpan}
We obtain
$$
Z_n^{\star}=M_n+\Delta_n^{1/4} N_n,
$$
where 
\begin{align*}
M_n=\frac{\Delta_n^{-1/4}}{\psi_2^n}\sum_{i=0}^{d_n-1} \left( \alpha_{t_{i k_n}}^2- \E[\alpha_{t_{i k_n}}^2| \F_{t_{i k_n}}]\right)
\mbox{ and } N_n=\sum_{k=1}^6 N_{n,k}+R_n
\end{align*}
 with 
\begin{align*}
N_{n,1}=&\frac{2 \Delta_n^{-1/2}}{\psi_2^n} \sum_{i=0}^{d_n-1} b^{[1]}_{t _{i k_n}} 
\int_{t _{i k_n}}^{t _{(i+1) k_n}} \nu_i^n(u) dW_u \\
N_{n,2}=& \frac{2 \Delta_n^{-1/2}}{\psi_2^n} \sum_{i=0}^{d_n-1} b_{t_{i k_n}}^{[1]} b_{t_{i k_n}}^{[2.1]} \int_{t_{i k_n}}^{t_{(i+1)k_n}}
\int_{t_{i k_n}}^{u} dW_s W(i,u)g_n(u-t_{i k_n})du, \\
N_{n,3}=& \frac{\Delta_n^{3/2} k_n^2 (\psi_3^n)^2}{\psi_2^n} \sum_{i=0}^{d_n-1} (b_{t_{i k_n}}^{[2]})^2, \\
N_{n,4}=&\frac{\Delta_n^{3/2} k_n^2 (2 \psi_4^n-\psi_2^n)}{2 \psi_2^n}\sum_{i=0}^{d_n-1} 2 b_{t_{i k_n} }^{[1]} b_{t_{i k_n}}^{[1.2]}+(b_{t_{i k_n}}^{[1.1]})^2\\
N_{n,5}=&\frac{2 \Delta_n^{-1/2}}{\psi_2^n} \sum_{i=0}^{d_n-1} \widebar \eps_{t_{i k_n}}  
\left( \psi_3^n k_n \Delta_n b_{t_{i k_n}}^{[2]}   +b_{t_{i k_n}}^{[1.1]} \int_{t_{i k_n}}^{t_{(i+1)k_n}}  \int_{t_{i k_n}}^{u}  dW_s dW(i,u)  \right), \\          
N_{n,6}=& \frac{\Delta_n^{-1/2} \psi_1^n d_n }{\psi_2^n k_n }\left[\omega^2 -\frac{\Delta_n }{2} \sum_{i=1}^{ 1/\Delta_n } (\Delta \eps_i^n)^2 \right],
\\ 
R_n =& o_{\p}(1)
\end{align*}
where
\begin{align*}
\nu_{i}^n(u)&=b_{t_{i k_n}}^{[1.1]} \left(\int_{t_{i k_n}}^{u} \left[2(W_s-W_{t_{i k_n}}) g_n(s-t_{i k_n})+W(i,s) \right] dW_s\right)  g_n(u-t_{i k_n}) \\ +& b_{t_{i k_n}}^{[1.1]} \left(\int_{t_{i k_n}}^{u} g_n(s-t_{i k_n}) ds g_n(u-t_{i k_n})
+\int_u^{t_{(i+1) k_n}} (g_n^2(s-t_{i k_n})-\psi_2^n)ds \right) \\
+& \psi_3^n k_n \Delta_n g_n(u-t_{i k_n}) b_{t_{i k_n}}^{[2]}  .
\end{align*}
\end{proposition}
\begin{proof} See Section \ref{sec7}.\end{proof} 
The meaning and the asymptotic behaviour of the quantity $(M_n,N_n)$ will be explained in Section \ref{sec5}.

\section{Asymptotic expansion theory with respect to mixed normal limit} \label{sec4}
In this section, we briefly summarize the main elements of the martingale expansion  
for a mixed normal limit, which was developed in \cite{YOSH13}.
Suppose that, on a filtered probability space $(\Omega, \mathcal{F}, (\mathcal{F}_t)_{t \in [0, 1]}, \p),$ we have a auxiliary random variable $Z_n$ satisfying $$Z_n=M_n+ r_n N_n,$$
where $N_n$ is a tight sequence of random variables and $(r_n)$ is 
a sequence of positive numbers satisfying $r_n \to 0$ (in our framework
$r_n=\Delta_n^{1/4}$). Note that we had this type of decomposition in the previous section. In addition, we assume that $M_n$ is a terminal value of some continuous $(\mathcal{F}_t)$-martingale $(M_t^n)_{t \in [0, 1]}$ with $M_0^n=0.$ 
We assume that $M_n$ (and hence $Z_n$) converges stably in law to a mixed normal limit $M:$
$$M_n \st M \sim MN(0, C).$$
Here, $M$ is defined on an extension $(\overline{\Omega}, \overline{\mathcal{F}}, \overline{\p})$ of 
$(\Omega, \mathcal{F}, \p).$
Let $F_n$ be a reference random variable, which is general here 
but will be $F_n$ of (\ref{fn}) in Section \ref{sec5}. 
We are interested in the asymptotic expansion of $(Z_n, F_n).$ 

Let $(C_t^n)_{t \in[0,1]}$ denote the quadratic variation process of $M^n$ 
and $(M_t)_{t \in[0,1]},$ defined on an extenstion $(\overline{\Omega}, \overline{\mathcal{F}}, \overline{\p})$ of $(\Omega, \mathcal{F}, \p),$ be a process satisfying $M=M_1.$ 
For $C_n:=C_1^n$ and $F_n$, we denote 
$$\widehat C_n=r_n^{-1}(C_n-C) \mbox{ and } \widehat F_n=r_n^{-1}(F_n-F),$$ 
respectively, where $C$ and $F$ are some random variables.
We impose the following crucial assumption where [B1](i) involves a functional stable convergence.
\bd
\item[[B1\!\!]] \textbf{(i)}
$(M^n_{\cdot}, N_n, \widehat C_n, \widehat F_n) \st (M_{\cdot}, N, \widehat C, \widehat F),$
\begin{itemize} 
 \item[\textbf{(ii)}]  $M_t \sim MN(0, C_t) \mbox{ for each } t \in [0,1].$
\end{itemize}
\ed

All information concerning the Edgeworth expansion for $(Z_n, F_n)$ is contained  in two random symbols $\underline{\sigma}$ and $\overline{\sigma}$, which are introduced in the next subsection.

\subsection{The random symbols $\underline{\sigma}$ and $\overline{\sigma}$}\label{subsec:RandomSymbols}
We start with the random symbol $\underline{\sigma}.$ Let $\widetilde{\mathcal{F}}=\mathcal{F} \vee \sigma(M).$ 
We assume that there exists random variables $\wtilde{C}(z),~  \wtilde{N}(z), ~\wtilde{F}(z)$ such that
$$\widetilde{C}(M)=\mathbb{E}[\widehat C|\widetilde{\mathcal{F}}],~~~
\widetilde{N}(M)=\mathbb{E}[N|\widetilde{\mathcal{F}}], ~~~
\widetilde{F}(M)=\mathbb{E}[\widehat F|\widetilde{\mathcal{F}}].$$
Then, we define the adaptive (classical) random symbol $\underline{\sigma}$ by 
\begin{equation} \label{sigmalower}
\underline{\sigma}(z, \im u, \im v)=\frac{(\im u)^2}{2} \wtilde{C}(z)+\im u \wtilde{N}(z)+\im v \wtilde{F}(z).
\end{equation}
The anticipative random symbol $\overline{\sigma}$ is more involved and 
only given implicitly. 
For this purpose, we define 
\begin{equation}\label{EqPhin}
\Phi_n(u,v)=\E[\exp(-u^2C/2 + \im vF)(e_1^n(u)-1) \psi_n]
\end{equation}
where $e_t^n(u)=\exp(\im u M_t^n+u^2 C_t^n/2)$ and $\psi_n$ is a truncation functional 
%that will be defined later. 
that takes values in $[0,1]$ satisfying at least 
(i) $\bbP[\psi_n=1]=1-o(r_n^{1+\kappa})$ as $n\to\infty$ for some positive constant $\kappa$, and 
(ii) $C_n-C$ is bounded whenever $\psi_n=1$. 
We observe that $(e_t^n(u))_{t\in [0,1]}$ is an exponential martingale, 
that is integrable under the truncation by $\psi_n$. 
Computations of $\Phi_n(u,v)$ can be done as if $\psi_n=1$ in practice 
since the effect of the truncation is asymptotically negligible. 

For $\alpha=(\alpha_1, \alpha_2) \in \mathbb{N}_0^2$, 
\footnote{$\mathbb{N}_0=\bbZ_+=\{0,1,2,...\}$.}
let $|\alpha|=\alpha_1+\alpha_2.$ 
For a function of two variables, we use the following differential operator notations:
$$d^\alpha=d^{\alpha_1}_{x_1} d^{\alpha_2}_{x_2} \mbox{ and } \partial^\alpha=\im^{-|\alpha|} d^\alpha.$$
Set $\Phi_n^\alpha(u,v)=\partial^{\alpha} \Phi_n(u,v)$.
We suppose that the limit 
$$\Phi^{\alpha}(u,v)=\lim_{n \to \infty} r_n^{-1} \Phi_n^\alpha(u,v)$$ 
exists and has the form
\bea\label{characterize.upperbar.general}
\Phi^{\alpha}(u,v)=\partial^{\alpha}\E[\exp(-u^2C/2 + \im vF) 
\overline{\sigma}(\im u, \im v)]
\eea
for every $\alpha\in\bbZ_+^2$,
where $\overline{\sigma}$ is given by
\begin{equation} \label{sigmaupper}
\overline{\sigma}=\sum_{j} \overline{c}_j (\im u)^{m_j} (\im v)^{n_j} \mbox{  (finite sum)}
\end{equation}
where $n_j, m_j \in \mathbb{N}$ and $\overline{c}_j$ are random variables. 
%Details are given in the assumption \textbf{(A4)} in Appendix.
See \cite{YOSH13} for details of the random symbols. 

\begin{remark} \label{SigmaUpperRem}
We note that the random symbol $\underline{\sigma}$ dates back to \cite{YOSH97} which deals with a martingale expansion associated with a normal limit. On the other hand, the random symbol $\overline{\sigma}$ first appeared in \cite{YOSH13} and is due to the mixed normality of the limit. Indeed, if $C$ is deterministic, we may pretend $\psi_n=1$ by a suitable stopping argument  
and obtain $\Psi_n(u,v)=0$ due to the martingale property of $e_t^n(u).$ That means $\overline{\sigma}=0.$
\end{remark}

\subsection{The asymptotic expansion of $(Z_n,F_n)$} \label{subsec:AsymptoticExpansion}
We define the full random symbol $\sigma$ by
$$\sigma=\underline{\sigma}+\overline{\sigma}.$$
We recall from \eqref{sigmalower} and \eqref{sigmaupper} that $\underline{\sigma}$ and $\overline{\sigma}$ are finite polynomials in 
$(\im u, \im v)$ with random coefficients. 
Hence, $\sigma$ admits the decomposition
$$\sigma(z, \im u, \im v)=\sum_{j} c_j(z) (\im u)^{m_j} (\im v)^{n_j} \mbox{  (finite sum)}$$
for some $n_j, m_j \in \mathbb{N}.$ 
We set the approximated density of $(Z_n, F_n)$ as
\begin{align}\label{density.general}
p_n(z,x)=&
\bbE[\phi(z;0,C)|F=x] p^{F}(x)\\
&+r_n \sum_{j}  (-d_z)^{m_j} (-d_x)^{n_j} 
\left( \E[c_j(z)\phi(z;0,C)|F=x]p^{F}(x) \right) \nonumber,
\end{align}
where $\phi(\cdot;0,b^2)$ and $p^{F}$ denote the densities of $N(0, b^2)$ and $F$, respectively. 
We note that  $p^{F}$ exists due to the condition 
that will be imposed later.  
For $K, \gamma>0,$ let 
$$\mathcal{E}(K, \gamma)= \left\{ h:\mathbb{R}^2 \to \mathbb{R}| h \mbox{ is measurable and }|h(z,x)| \leq K(1+|z|+|x|)^{\gamma} \right\}.$$
For $h \in \mathcal{E}(K, \gamma),$ we denote
$$\Delta_n(h)=\left|\E[h(Z_n, F_n)]-\int h(z,x) p_n(z,x) dz dx \right|.$$
We are now at the stage to recall a basic result. We need elements of
Malliavin calculus to state it; see e.g. the book \cite{Nual06} for the main concepts. 
In what follows, 
we will only treat one-dimensional functionals $M^n$ and $F_n$, and 
a two-dimensional Gaussian process as the input process, for simplicity of notation. 
However, it is sufficient for the purpose of this paper.
 
%On probability space $(\Omega,\calf,P)$, we suppose that 
For $\bbH=L^2([-1,1]\times\{1,2\},dt\times\nu)$, $\nu$ being the counting measure, 
let $w=(w(\sfh))_{\sfh\in\bbH}$ be a Gaussian process associated with the Hilbert space $\bbH$. 
That is, $w$ is a family of centered Gaussian random variables such that 
$\E[w({\sf h})w({\sf g})]=\int_{[-1,1]\times\{1,2\}}{\sf h}{\sf g}\>dxd\nu$ for ${\sf h},{\sf g}\in\bbH$. 
The Malliavin derivative is denoted by $D$, while its dual, also called divergence operator, is
denoted by $\delta=D^*$. For a separable Hilbert space ${\sf E}$, 
the Sobolev spaces $\bbD_{s,p}({\sf E})$ of ${\sf E}$-valued random variables are 
well defined, where $s$ is the index of differentiability and $p$ is the index of integrability. 
We simply write $\bbD_{s,p}$ for $\bbD_{s,p}(\bbR)$. 
Let $\bbD_{s,\infty}(E)=\bigcap_{p>1}\bbD_{s,p}(E)$. 

For a multivariate functional $U=(U_1,\ldots, U_d)$  
the Malliavin covariance matrix of  $U$ is given by 
$\sigma_U:=(\langle DU_i, DU_j \rangle_{\bbH})_{1\leq i,j\leq d}$. 
We also set  $\Delta_U:=\det\sigma_U$. 

Besides [B1], we will consider the following conditions. 
We note that the $\bbR$-valued functional $\xi_n$ appearing below is used to construct a truncation functional $\psi_n.$ 
%We will specify $\xi_n$ in Section \ref{SecKsi}.
%
\bd
\item[[B2\!\!]]$_\ell$ {\bf (i)} $F\in\bbD_{\ell,\infty}$ and $C\in\bbD_{\ell,\infty}$. 
\bd
\item[(ii)] $M_n\in\bbD_{\ell+1,\infty}$, $F_n\in\bbD_{\ell+1,\infty}$, 
$C_n\in\bbD_{\ell,\infty}$, $N_n\in\bbD_{\ell+1,\infty}$ and
$\xi_n\in\bbD_{\ell,\infty}.$
Moreover, for every $p>1,$
\begin{align*}
\sup_{n\in\bbN} &\bigg\{\|M_n\|_{\ell+1,p}+\|\hat{C}_n\|_{\ell+1,p}+\|\hat{F}_n\|_{\ell+1,p}
\\&+\|N_n\|_{\ell+1,p}+\|\xi_n\|_{\ell,p}\bigg\}<\infty
\end{align*}

\ed
\ed

\bd
\item[[B3\!\!]] {\bf (i)} $\lim_{n\to\infty}\p[|\xi_n|\leq1/2]=1$. 
\bd
\item[(ii)] $|C_n-C|>r_n^{1-a}$ implies $|\xi_n|\geq1$, where $a\in(0,1/3)$ is a constant. 
\item[(iii)] For any $p>1$, 
$\displaystyle \limsup_{n\to\infty}\bbE\big[1_{\{|\xi_n|\leq1\}}\Delta_{(M_n,F)}^{-p}\big]<\infty$
and $C^{-1}\in\bigcap_{p>1}L^p$. 
\ed
\ed

%Let $\ell_*=4$. 
\bd
\item[[B4\!\!]]$_{\ell,{\mathfrak m},{\mathfrak n}}$ 
{\bf (i)} $\underline{\sigma}$ is a random symbol of the form 
\beas 
\underline{\sigma}(z,{\tt i}u,{\tt i}v)
&=& 
\sum_j b_jz^{k_j}({\tt i}u)^{m_j}({\tt i}v)^{n_j} \qquad 
(b_j\in\bbD_{4,\infty},\ m_j\leq2,\ n_j\leq1)
\eeas
\bd
\item[(ii)] There exists a random symbol $\overline{\sigma}$ having a representation 
\beas 
\overline{\sigma}({\tt i}u, {\tt i}v) 
&=& \sum_j c_j({\tt i}u)^{m_j}({\tt i}v)^{n_j} \qquad
%\text{(finite sum)}
(c_j\in\bbD_{\ell,\infty},\ m_j\leq{\mathfrak m},\ n_j\leq{\mathfrak n})
\eeas
and (\ref{characterize.upperbar.general}) holds for every $\alpha\in\bbZ_+^2$. 
\ed
\ed
In assumption [B5], the term $\Phi_n$ (see \eqref{EqPhin}) involves the truncation functional $\psi_n$ which is described below. Suppose that $\psi:\bbR\to[0,1]$ in $C^\infty(\bbR)$ satisfies 
$\psi(x)=1$ if $|x|\leq1/2$ and $\psi(x)=0$ if $|x|\geq1$.
Let $Q_n=(M_n,F)$ and $R_n=(N_n,\hat{F}_n)$ and define 
a random matrix $R_n'$ by 
\beas 
R_n' 
&=& 
\sigma_{Q_n}^{-1}\big(r_n\langle DQ_n,DR_n\rangle_\bbH
+r_n\langle DR_n,QR_n\rangle_\bbH
+r_n^2\langle DR_n,DR_n\rangle_\bbH\big). 
\eeas
Then $\sigma_{(Z_n,C_n)}=\sigma_{Q_n}(I_2+R_n')$. Let 
\beas 
\xi_n' 
&=& 
r_n^{-1}|R_n'|^2. 
\eeas
Then, the truncation functional $\psi_n$ is composed by 
\bea\label{270819-2} 
\psi_n &=& \psi(\xi_n)\psi(\xi_n'). 
\eea
%where $\xi_n$ will be specified in Section \ref{SecKsi}. 
The functional $\xi_n$ will be set more concretely in Section \ref{SecKsi} for our application.

\bd
\item[[B5\!\!]] 
For every $\alpha\in\bbZ_+^2$ and some $\ep=\ep(\alpha)\in(0,1)$, 
\beas 
\limsup_{n\to\infty} \sup_{(u,v)\in\Lambda^0_n(2,q)}
r_n^{-1}|(u,v)|^{3-\ep}|\Phi^\alpha_n(u,v)| &<& \infty,
\eeas
where $\Lambda^0_n(2,q)=\{(u,v);\>|(u,v)|\leq r_n^{-q}\}$ and $q=(1-a)/2$. 
\ed
The following customizes Theorem 1 of \cite{YOSH13}. 

\begin{theorem} \label{YOSH13MT}
Let $\mathfrak{n}=\max_j n_j$ and $\ell=\max(5, 2 [(\mathfrak{n}+3)/2]).$ 
Let $K$, $\gamma\in(0,\infty)$ and $\kappa\in(0,1)$ be arbitrary numbers.
Suppose that Conditions $[B1]$, $[B2]_\ell$, $[B3]$, 
$[B4]_{\ell,{\mathfrak m},{\mathfrak n}}$ and $[B5]$ are satisfied.
Then for some constant $K_1=K_1(K,\gamma,\kappa)$,
\bea\label{270819-1} 
\sup_{h \in \mathcal{E}(K, \gamma)} \Delta_n(h)
&\leq&
K_1\bbP\big[|\xi_n|>1/2\big]^\kappa+o(r_n).
\eea
\end{theorem}

In other words, $p_n$ is the second order Edgeworth expansion of the 
distribution of the pair $(Z_n,F_n)$, 
if the event truncated by $\xi_n$ is sufficiently small.

See \cite{YOSH13} for details of the above theorem 
and other information, and also arXiv 1210.3680v3 for updates. 
The Malliavin calculus is used to derive the asymptotic expansion formula $p_n$. 
Further, one needs the non-degeneracy of the Malliavin covariance since 
the problem of validity of the asymptotic expansion is deeply related with 
the regularity of the distribution of the underlying functional. 
There is a counterexample, even in the classical expansion, if [B3] (iii) is not satisfied. 
Condition [B5] is also a requirement for the non-degeneracy of the same kind but regarding 
the correction term corresponding to the anticipative random symbol. 
\begin{remark}\rm 
Under a stronger assumption that $\bbP[|\xi_n|\leq1/2]=1-O(r_n^\kappa)$ for any $\kappa>0$, 
in place of [B3] (i), 
we can simply use $\psi_n=\psi(\xi_n)$ without $\xi_n'$ apparently, 
and remove the first term on the right-hand side of (\ref{270819-1}). 
This makes presentation of the result slightly simpler though 
the deeper truncation (\ref{270819-2}) is re-constructed in the proof. 
In this case, Condition [B5] may also become stronger since the truncation reduces. 
\end{remark}

\section{Main results: Asymptotic expansion for the pre-averaging estimator} \label{sec5}
In this section, we utilise the results from the general theory for a mixed normal limit and obtain the Edgeworth expansion for the pre-averaging estimator.

\subsection{Assumptions}
We will consider $F_n$ defined in (\ref{fn}) as a consistent estimator of $C$. 
We denote by $C_{b}^\infty$ the set of smooth functions on $\bbR$ whose all derivatives of positive orders 
are bounded. 
Let 
\beas 
a(x) &=&  2\theta \left((b^{[1]}(x))^2+ \frac{\omega^2 \psi_1}{\theta^2 \psi_2} \right)^2.
\eeas
We are assuming that $\text{supp}\call\{X_0\}$ is compact and that $\omega$ is positive. 
We impose the following condition on the processes $b^{[1]}$ and $b^{[2]}$.
\bd
\item[[V\!\!]] 
{\bf (i)} $b^{[1]}$, $b^{[2]}\in C^\infty_{b}$ and $b^{[1]}(x)\not=0$ 
for $x\in\text{supp }\call\{X_0\}$. 
\bd
\item[(ii)] $\sum_{k=1}^\infty |d_x^ka(X_0)|>0$. 
\ed\ed

\begin{remark}\rm 
If $b^{[1]}$ is nonnegative, Condition [V] (ii) can be replaced by 
\bd
\item[(ii$'$)]  $\sum_{k=1}^\infty |d_x^kb^{[1]}(X_0)|>0$. 
\ed
\end{remark}

\begin{remark}\rm 
By assumption, $\text{supp }\call\{X_0\}$ is compact.
We do not assume uniform ellipticity of the diffusion on the whole domain of $b^{[1]}$.
The microstructure noise serves as a smoother of the distribution of $M_n$. 
On the other hand, we need regularity of the distribution of $C$ defined in Theorem \ref{CLT}. 
Practically this would be satisfied once $C$ is random. 
If $C$ is deterministic, the problem of asymptotic expansion becomes a classical one 
that is tractable by \cite{YOSH97}.  We refer to Section 6 for the exposition of this setting. 
\end{remark}

\subsection{Stable limit theorem}\label{SLT}
We have seen in the previous section that the random coefficients of $\underline \sigma$ solely depend on the limit of the stable convergence found in the condition \textbf{[B1]}. Hence, we need to compute $M, N, \widehat{C}, \widehat{F}.$ 
In this section, we assume that 
\bea\label{20150816-1} 
\eps_{t_i}=\omega \Delta_n^{-1/2}  (B_{t_i}-B_{t_{i-1}})\qquad (i\geq0)
\eea
where 
for the Gaussian process $w=(w(\sfh))_{\sfh\in\bbH}$, 
$B_t=w(\mathbbm{1}_{[-1,t]\times\{2\}})$ for $t\in[-1,1]$, 
as well as $W_t=w(\mathbbm{1}_{[0,t]\times\{1\}})$ for $t\in[0,1]$. 
We assume that $w$ is independent of $X_0$. The Malliavin derivative in the directions of $W$ and $B$ are denoted by $D^{(1)}$  and $D^{(2)}$, respectively.

The particular Gaussian framework 
of the model \eqref{20150816-1} is imposed to be able to use Malliavin calculus.   
Our results can be directly extended to a more general setting
\[
\eps_{t_i}= \Delta_n^{-1/2} \int_{t_{i-1}}^{t_i} \omega_s dB_s,
\]
where $\omega$ is adapted to the filtration $\mathcal G_t=\sigma(X_0, (W_s)_{s\leq t})$, under mild assumptions on the stochastic process $(\omega_t)_{t\geq 0}$ (cf. \cite{JLMP09}). However, we dispense with the detailed exposition of this case.   

Similarly to $g_n(s)$ and $W(i, t),$ we define 
\begin{align*}
h_n(s)&=\sum_{j=0}^{k_n-1} h(j/k_n) 1_{(t_{j-1}, t_j]}(s), 
\\
\eps(i, t)&=-\omega \Delta_n^{-1/2} \int_{t_{i k_n-1}\wedge t}
^{t_{(i+1) k_n-1}\wedge t} h_n(u-t_{i k_n}) dB_u.
\end{align*}
We remark that 
$$\widebar \eps_{t_{i k_n}}=\eps(i, t_{(i+1) k_n})
=\ep(i,t_{(i+1)k_n-1})
.$$
Let
\beas 
\alpha(i, u)&=& b_{t_{i k_n}}^{[1]} W(i,u)+ \eps(i, u)
\eeas
and remark that $\alpha(i, u)$, conditionally on $\mathcal{F}_{t_{i k_n}},$ is distributed as 
$$N \left(0, \int_{t_{i k_n-1}}^{u} [g_n(u-t_{i k_n})^2 (b_{t_{i k_n}}^{[1]})^2 + \omega^2 \Delta_n^{-1} h_n(u-t_{i k_n})^2] du
\right).$$ 
We will consider the filtration ${\bf F}=(\calf_t)_{t\in[0,1]}$, each $\calf_t$ 
being generated by $X_0$, 
$\{W_s\}_{s\in[0,t]}$ and $\{B_s\}_{s\in[-1,t]}$.
It\^o's lemma implies that $M_n$ is the terminal value of the ${\bf F}$-continuous martingale 

\bea\label{270814-3}
M_t^n=\frac{2 \Delta_n^{-1/4}}{\psi_2^n} \sum_{i=0}^{d_n-1} \int_{t_{i k_n} \wedge t}^{t_{(i+1) k_n} \wedge t} 
\alpha(i, u)  d \alpha(i, u).  
\eea
According to the expression (\ref{270814-3}),
we observe that the quadratic variation of $M^n$ satisfies

\beas
C_n=\left \langle M^n \right \rangle_1
&=&
\frac{4 \Delta_n^{-1/2}}{(\psi_2^n)^2} \sum_{i=0}^{d_n-1} 
\int_{t_{i k_n}}^{t_{(i+1) k_n}}
\alpha^2 \left(i, u \right) [d \alpha(i, u)]^2, 
\eeas
where 
\begin{align} \label{dalpha}
[d \alpha(i, u)]^2=[(b_{t_{i k_n}}^{[1]})^2 g_n(u-t_{i k_n})^2+ \omega^2 \Delta_n^{-1} h_n(u-t_{i k_n})^2] du.
\end{align}
Before going to the stable limit theorem, we observe that some of terms included in $N_n$ converge in probability. We want to separate them from others.
\begin{lemma} We obtain the convergence in probability
$$N_{n,k} \cp N_k, ~ k=2,3,4$$
where the quantities $N_{n,k}$ are defined in Proposition \ref{StochExpan} and
\begin{align*}
N_2=& \frac{ \theta (\psi_3)^2}{ \psi_2} \int_0^1 b_u^{[1]} b_u^{[2.1]} du,\\
N_3=& \frac{ \theta (\psi_3)^2}{ \psi_2} \int_0^1 (b_u^{[2]})^2 du,\\
N_4=& \frac{ \theta (2 \psi_4-\psi_2)}{ 2 \psi_2} \int_0^1 2 b_u^{[1]} b_u^{[1.2]}+(b_u^{[1.1]})^2 du, \\
\end{align*}
\end{lemma}
After recalling the following notation used in the previous section
$$\widehat C_n=\Delta_n^{-1/4}(C_n-C), ~ \widehat F_n=\Delta_n^{-1/4}(F_n-C),$$
we are ready to state our stable convergence result.
\begin{theorem}\label{CLT2} We obtain the stable convergence
$$\left(M_n, N_n, \widehat C_n , \widehat F_n   \right) \st (M, N, \widehat C, \widehat F) \sim MN \left(\mu,\int_0^1 \Sigma_s ds \right),$$
where $\mu_1=\mu_3=\mu_4=\Sigma_s^{12}=\Sigma_s^{23}=\Sigma_s^{24}=0,$
$$~\mu_2= \int_0^1 \mu_2(b_s^{[1]}, b_s^{[2]}, b_s^{[1.1]}) dW_s+\sum_{k=2}^4 N_k$$ 
and
\begin{align*}
\Sigma_s^{11}&= 2 \theta \left((b_s^{[1]})^2+ \frac{\omega^2 \psi_1}{\theta^2 \psi_2} \right)^2, ~~
\Sigma_s^{22}=\sigma_2(b_s^{[1]}, b_s^{[2]}, b_s^{[1.1]})-\left[\mu_2(b_s^{[1]}, b_s^{[2]}, b_s^{[1.1]})\right]^2  
\\
\Sigma_s^{33}&= \frac{16 \theta^3}{3} \left((b_s^{[1]})^2+ \frac{\omega^2 \psi_1}{\theta^2 \psi_2} \right)^4, ~~
\Sigma_s^{44}= \frac{128}{3} \theta^3 \left((b_s^{[1]})^2+ \frac{\omega^2 \psi_1}{\theta^2 \psi_2} \right)^4 \\
\Sigma_s^{13}&= \frac{8  \theta^2}{3} \left((b_s^{[1]})^2+ \frac{\omega^2 \psi_1}{\theta^2 \psi_2} \right)^3,   ~~~
\Sigma_s^{14}= 8  \theta^2\left((b_s^{[1]})^2+ \frac{\omega^2 \psi_1}{\theta^2 \psi_2} \right)^3  ~~~ \\
\Sigma_s^{34}&= \frac{44  \theta^2}{3} \left((b_s^{[1]})^2+ \frac{\omega^2 \psi_1}{\theta^2 \psi_2} \right)^3,
\end{align*}
with
\begin{align*}
\mu_2(x,y,z)&= \frac{\theta x}{\psi_2} \left([(\psi_3)^2+2 \psi_4-\psi_2] z+2 (\psi_3)^2 y \right) \\
\sigma_{2}(x,y,z)&=\frac{4 \theta^2 x^2}{\psi_2^2}\left[  
(\psi_7+ \psi_6) z^2+(4 \psi_5-\psi_2) \psi_3 y z+ \psi_3^2 \psi_2 y^2 \right] 
\\ & +\frac{4 \omega^2 \psi_1}{\psi_2^2}\left[\psi_3^2 y^2+ \psi_4 z^2 \right]+\frac{3 \omega^4 \psi_1^2 }{\theta^4 \psi_2^2 }.
\end{align*}
\end{theorem}
\begin{proof} See Section \ref{sec7}. \end{proof}
\subsection{Computation of $\underline \sigma$ and $\overline \sigma$ }
It is now straightword to compute $\underline \sigma.$ Indeed, the mixed normality of the limit in Theorem \ref{CLT2} and \eqref{sigmalower} imply
\begin{equation} \label{sigmalower2}
\underline \sigma(z, \im u, \im v)= (\im u)^2 \mathcal{H}_1(z)+\im u \mathcal{H}_2+\im v \mathcal{H}_3(z)
\end{equation}
where
$$\mathcal{H}_1(z)=z \frac{\int_{0}^{1} \Sigma_{s}^{13}ds}{2 \int_{0}^{1} \Sigma_{s}^{11}ds}, ~~~ \mathcal{H}_2=\mu_2,~~~ \mathcal{H}_3(z)=z \frac{\int_{0}^{1} \Sigma_{s}^{14}ds}{ \int_{0}^{1} \Sigma_{s}^{11}ds}.$$
Now, we pass to the calculation of $\overline \sigma.$ 
The anticipative random symbol $\overline{\sigma}$ in (\ref{characterize.upperbar.general}) 
is characterized by  
\bea\label{characterize.upperbar}
\Phi^{\alpha}(u,v)=\partial^{\alpha}\E[\exp((-u^2/2 + \im v) C) \overline{\sigma}(\im u, \im v)]
\eea
in the present situation.
Using techniques from  Malliavin calculus, we obtain the following result.

\begin{proposition} \label{UpperSigmaLemma}
We obtain the identity
\begin{equation}\label{sigmaupper2}
\overline \sigma (\im u, \im v)= \im u \left(-\frac{u^2}{2}+ \im v \right)^2 \mathcal{H}_4+\im u \left(-\frac{u^2}{2}+ \im v\right) \mathcal{H}_5 ,
\end{equation}
where $c(x)=\left[(b^{[1]})^2(x)+ \frac{\omega^2 \psi_1}{\theta^2 \psi_2} \right]^2$ and
\begin{align*}
\mathcal{H}_4=& \frac{4 \theta^3 \psi_3^2 }{\psi_2} \int_0^1 (b^{[1]})^2(X_t) 
\left( \int_t^1 c'(X_r) D_t ^{(1)}X_r dr  \right)^2 dt, \\
\mathcal{H}_5=& \frac{2 \theta^2 \psi_3^2 }{\psi_2} \int_0^1 (b^{[1]})^2(X_t) \left( \int_t^1 
\left[c''(X_r) (D_t^{(1)} X_r)^2+c'(X_r) D_t^{(1)} D_t^{(1)} X_r \right] dr  \right) dt.
\end{align*}
\end{proposition}
\begin{proof} See Section \ref{sec7}. \end{proof}

\subsection{The asymptotic expansion of the pre-averaging estimator}
In view of \eqref{sigmalower2} and \eqref{sigmaupper2}, we observe that the full random symbol $\sigma=\underline \sigma+\overline \sigma$ is given by
\begin{equation}\label{sigma2}
\sigma(z, \im u, \im v)=\sum_{j=1}^8 c_j(z) (\im u)^{m_j} (\im v)^{n_j}
\end{equation}
where
\begin{align*}
m_1&=1,~~  n_1=0,~~  c_1(z)=\mathcal{H}_2,~~~~~~ m_2=0,~~ n_2=1,~~ c_2(z)=\mathcal{H}_3(z), \\
m_3&=2,~~  n_3=0,~~  c_3(z)=\mathcal{H}_1(z),~~ m_4=1,~~ n_4=1,~~ c_4(z)=\mathcal{H}_5,  \\
m_5&=3,~~  n_5=0, ~~ c_5(z)=\frac{1}{2} \mathcal{H}_5, ~~~~m_6=1,~~ n_6=2,~~ c_6(z)=\mathcal{H}_4, \\
m_7&=3, ~~ n_7=1, ~~ c_7(z)=\mathcal{H}_4,            ~~~~~~m_8=5,~~ n_8=0, ~~c_8(z)=\frac{1}{4} \mathcal{H}_4.
\end{align*}
We continue as in Section \ref{subsec:AsymptoticExpansion} and define the density $p_n(z,x)$ by 
\begin{align}
p_n(z,x)=&\phi(z;0,x) p^C(x)\\
&+\Delta_n^{1/4} \sum_{j=1}^8  (-d_z)^{m_j} (-d_x)^{n_j} 
\left(\phi(z;0,x)p^C(x) \E[c_j(z)|C=x] \right) \nonumber,
\end{align}
according to (\ref{density.general}).
For $h \in \mathcal{E}(K, \gamma),$ we also recall the notation
$$\Delta_n(h)=\left|\E[h(Z_n^{\star}, F_n)]-\int h(z,x) p_n(z,x) dz dx \right|.$$
The following theorem is the main result of this article. 

\begin{theorem}\label{MainTh} 
Suppose that Condition $[V]$ is satisfied.
Let $K>0, \gamma>0$. 
Then
$$\sup_{h \in \mathcal{E}(K, \gamma)} \Delta_n(h)=o(\Delta_n^{1/4})$$
\end{theorem}
\begin{proof} See Section \ref{sec7}. \end{proof}
This theorem is not the end of the story. From the point of view of statistical applications, the Edgeworth expansion associated to the studentized statistic 
$Z_n^{\star}/\sqrt{F_n}$ is more interesting.
Since the representation of $\sigma$ in \eqref{sigma2} is the same as in \cite[Section 6]{POYO13}, we easily obtain the second order Edgeworth expansion of  $Z_n^{\star}/\sqrt{F_n}.$

\begin{corollary}\label{MainCor}
Suppose that Condition $[V]$ is fulfilled. We define the random variables $\widetilde{\mathcal{H}}_1$ and
$\widetilde{\mathcal{H}}_3$ via the identity $\mathcal{H}_k(z)=z\widetilde{\mathcal{H}}_k$, $k=1,3$.
Then the second order Edgeworth expansion of  $Z_n^{\star}/\sqrt{F_n}$ is given by
\begin{align*}
 &p^{Z_n^{\star}/\sqrt{F_n}}(y)=\phi(y;0,1)+\Delta_n^{1/4} \phi(y;0,1) \Big[    y\Big( \E[\mathcal{H}_2 C^{-1/2}]-\frac{1}{2} 
\E[ \mathcal{H}_5 C^{-3/2}] \\
&+ \frac{3}{4} \E[ \mathcal{H}_4 C^{-5/2}] +\E[\widetilde{\mathcal{H}}_3 C^{-1/2}]-3 
\E[ \widetilde {\mathcal{H}}_1 C^{-1/2}]   \Big)\\
&+y^3\Big( \E[\widetilde{ \mathcal{H}}_1 C^{-1/2}]-\frac{1}{2}\E[\widetilde{ \mathcal{H}}_3 C^{-1/2}] \Big)   \Big] 
\end{align*}
\end{corollary}
Note that the polynomial involved in the second order term is odd of order $3$. However, in general
it is not connected to the third order Hermite polynomial, which appears in the classical Edgeworth expansion 
in the framework of i.i.d. observations, see e.g. Theorem 2.5 in \cite{Hall92}.

\section{The case of constant volatility}
The main focus of this paper was to investigate asymptotic expansions when the estimated 
object $C$ is random, as seen in the previous sections. We remark however that Condition $[V]$(ii)
is not satisfied when $b^{[1]}(x)=b$ for all $x$. In particular, the asymptotic expansion of Corollary
\ref{MainCor} can not be directly applied to the case of constant volatility.

For the sake of completeness, we thus present the second order Edgeworth expansion in the setting
$b^{[1]}(x)=b$ identically, which relies on an earlier work \cite{YOSH97}. This article studies asymptotic expansions associated with a classical central limit theorem and it does not require 
Condition $[V]$(ii).  We note that the expression for the asymptotic density   
$p^{Z_n^{\star}/\sqrt{F_n}}$ simplifies quite a bit in the case of constant volatility. In particular, 
in view of  Remark \ref{SigmaUpperRem}, we obtain that $\overline{\sigma}=0$. Hence, 
$\mathcal{H}_4=\mathcal{H}_5=0$. The following version of Corollary \ref{MainCor} 
is a consequence of \cite[Theorem 1]{YOSH97}.

\begin{theorem}\label{ThmDeterministic}
Suppose that $b^{[1]}(x)=b$ identically and Condition $[V]$(i) holds. 
Then the second order Edgeworth expansion of  $Z_n^{\star}/\sqrt{F_n}$ is given by
\begin{align*}
 p^{Z_n^{\star}/\sqrt{F_n}}(y)&=\phi(y;0,1)\\
&+\Delta_n^{1/4} \phi(y;0,1) C^{-1/2} 
\Big[    y\Big( \E[\mathcal{H}_2]+\widetilde{\mathcal{H}}_3-3 \widetilde {\mathcal{H}}_1   \Big)
+y^3\Big( \widetilde{ \mathcal{H}}_1-\frac{1}{2}\widetilde{ \mathcal{H}}_3  \Big)   \Big].
\end{align*}
\end{theorem}
\begin{proof} 
First, we notice that Condition $[V]$(i) implies the assumptions of  \cite[Theorem 1]{YOSH97}.
The asymptotic expansion of \cite[Theorem 1]{YOSH97} has not been obtained for the pair
$(Z_n^{\star},F_n)$, hence we require a stochastic expansion for the studentized statistic  
$Z_n^{\star}/\sqrt{F_n}$ directly. Denoting $r_n=\Delta_n^{1/4},$ the Taylor expansion yields
\begin{align*}
\frac{Z_n^{\star}}{\sqrt{F_n}}&=(M_n+r_n N_n)\left(\frac{1}{C^{1/2}}-\frac{(F_n-C)}{2 C^{3/2}} +o_{\p}(r_n)\right) \\
     &=:\mathcal{M}_n+r_n \mathcal{N}_n,
\end{align*}
where, recalling the notation $\widehat{F}_n=r_n^{-1}(F_n-C),$ we have
\begin{align*}
\mathcal{M}_n=\frac{M_n}{C^{1/2}} \quad \mbox{ and } \quad
\mathcal{N}_n= \frac{N_n}{C^{1/2}}-\frac{M_n \widehat{F}_n}{2 C^{3/2}}+o_{\p}(1).
\end{align*}     
Let us denote  $\mathcal{C}_n:= C_n/C$. 
Applying Theorem \ref{CLT2}, we deduce the joint stable convergence
\begin{equation*}
(\mathcal{M}_n, \mathcal{N}_n, r_n^{-1}(\mathcal{C}_n-1)) \st 
\left(\frac{M}{C^{1/2}}, \frac{N}{C^{1/2}}-\frac{M \widehat{F} }{2 C^{3/2}}, \frac{\widehat C}{C} \right)=:(\mathcal{M}, \eta, \xi),
\end{equation*}
where $\xi$ and $\eta$ follow the notation from Theorem 1 of \cite{YOSH97}. Recalling 
equations \eqref{sigmalower} and \eqref{sigmalower2}, we observe the identities 
\begin{align*}
\mathbb{E}[\eta | \mathcal{M}=z]= \frac{\E[\mathcal{H}_2]}{C^{1/2}}-\frac{z^2 \widetilde{\mathcal{H}}_3}{2 C^{1/2}} \quad \text{and} \quad
\mathbb{E}[\xi | \mathcal{M}=z]= \frac{2 z \widetilde{\mathcal{H}}_1}{C^{1/2}}.
\end{align*}
The second order Edgeworth expansion of \cite[Theorem 1]{YOSH97}  implies the formula 
$$p^{Z_n^{\star}/\sqrt{F_n}}(y)=\phi(y;0,1)+\frac{1}{2} r_n \partial_y^2 (\mathbb{E}[\xi | \mathcal{M}=y] \phi(y;0,1))-r_n \partial_y (\mathbb{E}[\eta | \mathcal{M}=y].$$
A straightforward calculation implies the desired asymptotic expansion. 
\end{proof}

\section{Example} \label{sec6}
\begin{example} Let $a>0$ and $\sigma>0.$ We consider the Black-Scholes model 
to illustrate the computations of the previous sections: 
$$dX_t= a X_t dt+\sigma X_t dW_t.$$
In this framework we have that 
\begin{align*}
b_t^{[1]}=&\sigma X_t, ~~ b_t^{[1.1]}=\sigma^2 X_t,~~ b_t^{[1.2]}=a \sigma  X_t, \\
b_t^{[2]}=& a X_t, ~~ b_t^{[2.1]}=a \sigma X_t, ~~b_t^{[2.2]}=a^2 X_t.
\end{align*}
Then, we immediately obtain that
\begin{align*}
C=&2 \theta \int_0^1 \left(\sigma^2 X_t^2+\frac{\omega^2 \psi_1}{\theta^2 \psi_2} \right)^2 dt, ~
\widetilde{\mathcal{H}}_1=\frac{ 2 \theta}{3} \frac{\int_0^1 \left(\sigma^2 X_t^2+\frac{\omega^2 \psi_1}{\theta^2 \psi_2} \right)^3 dt}{
\int_0^1 \left(\sigma^2 X_t^2+\frac{\omega^2 \psi_1}{\theta^2 \psi_2} \right)^2 dt}, ~
\widetilde{\mathcal{H}}_3=6 \widetilde{\mathcal{H}}_1,  \\
\mathcal{H}_2=& \frac{\theta \sigma[(\psi_3^2+2 \psi_4-\psi_2)\sigma^2+ 2\psi_3^2 a]}{\psi_2} \int_0^1 X_t^2 dW_t \\
&+\frac{\theta[2 \psi_3^2(\sigma^2 a+a^2)+(2 \psi_4-\psi_2)(2 \sigma^2 a+\sigma^4)] }{2 \psi_2} \int_0^1  X_t^2 dt.
\end{align*}
We observe that $D_s^{(1)} X_t$ satisfies, for $t \geq s,$
$$D_s^{(1)} X_t= a D_s^{(1)} X_t dt+\sigma D_s ^{(1)} X_t d W_t, \quad D_s^{(1)} X_s=\sigma X_s.$$
This easily implies
$$D_s^{(1)} X_t=\sigma X_s \exp \left[ (a-\sigma^2/2)(t-s)+\sigma (W_t-W_s)\right]=\sigma X_t.$$
Hence, we deduce that $c(x)=\left[\sigma^2x^2+ \frac{\omega^2 \psi_1}{\theta^2 \psi_2} \right]^2$ and
\begin{align*}
\mathcal{H}_4=& \frac{4 \theta^3 \psi_3^2 \sigma^4}{\psi_2} \int_0^1 X_t^2 
\left( \int_t^1 c'(X_r) X_r dr  \right)^2 dt, \\
\mathcal{H}_5=& \frac{2 \theta^2 \psi_3^2 \sigma^4}{\psi_2} \int_0^1 X_t^2 \left( \int_t^1 
\left[c''(X_r) X_r^2+c'(X_r) X_r \right] dr  \right) dt.
\end{align*}
Using the above quantities 
we may obtain the second order Edgeworth expansion of $Z_n^{\star}/\sqrt{F_n}$ using Corollary \ref{MainCor}.
\end{example}

%\begin{example} Let $\kappa>0$ and $\sigma>0.$ The Ornstein-Uhlenbeck process satisfies 
%$$dX_t= \kappa(\mu-X_t)dt+\sigma dW_t.$$
%This means 
%\begin{align*}
%b^{[1]}=&\sigma, ~~ b^{[1.1]}=0,~~ b^{[1.2]}=0, \\
%b^{[2]}=&\kappa(\mu-X_t), ~~ b^{[1.1]}=-\kappa \sigma, ~~b^{[1.2]}=-\kappa^2(\mu-X_t).
%\end{align*}
%Plugging the quantities, we obtain $\mathcal{H}_4=\mathcal{H}_5=0$ and
%\begin{align*}
%C=&2 \theta \left(\sigma^2+\frac{\omega^2 \psi_1}{\theta^2 \psi_2} \right)^2, ~~
%\widetilde{\mathcal{H}}_1=\frac{ 2 \theta \left(\sigma^2+\frac{\omega^2 \psi_1}{\theta^2 \psi_2} \right)}{3}, ~~
%\widetilde{\mathcal{H}}_3=4 \theta \left(\sigma^2+\frac{\omega^2 \psi_1}{\theta^2 \psi_2} \right),  \\
%\mathcal{H}_2=& \frac{2 \theta (\psi_3)^2}{\psi_2} \int_0^1 \left(  b_u^{[1]}b_u^{[2]}\right) dW_u+
%\frac{\theta (\psi_3)^2}{\psi_2} \int_0^1  \left(  b_u^{[1]}b_u^{[2.1]}+(b_u^{[2]})^2\right) du.
%\end{align*}
%Hence, the second order Edgeworth expansion becomes
%\begin{align*}
%p_n(y)&=\phi(y;0,1) +\Delta_n^{1/4} \phi(y;0,1) \Big[    y\Big( \E[\mathcal{H}_2 C^{-1/2}] +
%\E[\widetilde{\mathcal{H}}_3 C^{-1/2}]-3 
%\E[ \widetilde {\mathcal{H}}_1 C^{-1/2}]   \Big)\\
%&+y^3\Big( \E[\widetilde{ \mathcal{H}}_1 C^{-1/2}]-\frac{1}{2}\E[\widetilde{ \mathcal{H}}_3 C^{-1/2}] \Big)   \Big] \\
%&=\phi(y;0,1) +\Delta_n^{1/4} \phi(y;0,1) \Big[    y\Big( \frac{\theta (\psi_3)^2}{\psi_2} \frac{\sigma^2(1-4 \kappa+e^{-2\kappa})+\kappa \mu^2(1-e^{-4 \kappa})}{4 C^{1/2}}+ \sqrt{2 \theta} \Big) \\
%& - y^3 \frac{2 \sqrt{2 \theta}}{3}   \Big].\\
%\end{align*}
%\end{example}

\section{Proofs} \label{sec7}
%\begin{proof}[Sketch of the proof of Theorem \ref{LLN}]
%The first order approximation of $\widebar{Y} _{t _{i k_n}}$ is given by the process
%$\alpha_{t_{i k_n}}=b^{[1]}_{t _{i k_n}} \widebar{W} _{t _{i k_n}}+\widebar \eps _{t _{i k_n}}.$
%A calculation shows that 
%\begin{align*}
%\frac{1}{\psi_2^n} \sum_{i=0}^{d_n-1} \E[(\alpha_{t_{i k_n}})^2|\mathcal{F}_{t _{i k_n}}]
%=&\frac{1}{\psi_2^n} \sum_{i=0}^{d_n-1} \left[(b^{[1]}_{t _{i k_n}})^2 \psi_2^n k_n \Delta_n +\frac{\psi_1^n \omega^2}{k_n} \right] \\
%=&\sum_{i=0}^{d_n-1} \left[(b^{[1]}_{t _{i k_n}})^2 k_n \Delta_n \left]+ \frac{\psi_1^n \omega^2}{k_n^2 \Delta_n \psi_2}.
%\end{align*}
%Hence, we obtain a Riemann approximation of $V$ and a bias term involving $\omega^2.$  To remove this bias term, we use the statistic
%$$\frac{\Delta_n}{2} \sum_{i=1}^{ 1/\Delta_n } (\Delta Y_i^n)^2$$ which is a consistent estimator of $\omega^2$ (see \cite{ZHMA05}).
%\end{proof}

\subsection{Sketch of the proof of Theorem \ref{CLT}}\label{sketch.clt}
%\begin{proof}[Sketch of the proof of Theorem \ref{CLT}]
Without loss of generality, we suppose that the processes $b^{[1]}$ and $b^{[2]}$ are bounded.
This is done following a standard localization procedure, see \cite{BGJPS06} for details.
The first order approximation of $\widebar{Y} _{t _{i k_n}}$ is given by the process
$\alpha_{t_{i k_n}}=b^{[1]}_{t _{i k_n}} \widebar{W} _{t _{i k_n}}+\widebar \eps _{t _{i k_n}}$ (see also the statement of Proposition \ref{StochExpan}).
Hence, we obtain that the dominating term in $Z_n$ is
\beas 
M_n 
&=& 
\frac{\Delta_n^{-1/4}}{\psi_2^n}\sum_{i=0}^{d_n-1} 
\bigg\{
\alpha_{t_{i k_n}}^2- \E[\alpha_{t_{i k_n}}^2| \F_{t_{i k_n}}]
\bigg\}.
\eeas
%%
%We note that the squre-integrability of $\alpha_{t_{ik_n}}$ is implicitly assumed but 
%it can be validated by a suitable localization procedure. 
%%
Using the notation $$\beta_{t_{i k_n}}=\frac{\Delta_n^{-1/4}}{\psi_2^n} \left(\alpha_{t_{i k_n}}^2- \E[\alpha_{t_{i k_n}}^2| \F_{t_{i k_n}}]\right),$$
we observe that $\beta_{t_{i k_n}}$ is $\mathcal{F}_{t_{(i+1) k_n}}$-measurable and 
$\E[\beta_{t_{i k_n}}|\mathcal{F}_{t_{i k_n}}]=0$ holds. Moreover, we get
$$\sum_{i=0}^{d_n-1} \E[\left(\beta_{t_{i k_n}} \right)^2|\mathcal{F}_{t_{i k_n}}]=
\frac{2 \Delta_n^{-1/2}}{(\psi_2^n)^2} \sum_{i=0}^{d_n-1}  \left((b^{[1]}_{t _{i k_n}})^2 \psi_2^n k_n \Delta_n +\frac{\psi_1^n \omega^2}{k_n} \right)^2 \cp C.$$
Therefore, the first claim follows from Theorem IX.7.28 of \cite{JASH03}. As for the second claim, we again observe that the first order approximation of $\widebar{Y} _{t _{i k_n}}$ is given by the process
$\alpha_{t_{i k_n}}.$ Moreover, we see that the main term in $F_n$ satisfies
$$\frac{2 \Delta_n^{-1/2}}{3 (\psi_2^n)^2} \sum_{i=0}^{d_n-1} \E[\left(\alpha_{t_{i k_n}} \right)^4|\mathcal{F}_{t_{i k_n}}]
=2\Delta_n^{-1/2} \sum_{i=0}^{d_n-1}  
\left((b^{[1]}_{t _{i k_n}})^2 k_n \Delta_n +\frac{\psi_1^n \omega^2}{\psi_2^nk_n} \right)^2
\cp C.$$
\qed

\subsection{Proofs of Proposition \ref{StochExpan} and Theorem \ref{CLT2}}
\begin{proof}[Proof of Proposition \ref{StochExpan}]
Proceeding as the in previous section, we apply a localization procedure and suppose that all processes of the form $b^{[k_1 \ldots k_m]},$ $k_i=1,2,$ are bounded.
We will apply the following version of Burkholder's inequality several times. For any process $U$ as in \eqref{SDE} with bounded drift and  diffusion terms and 
any $p \geq 0,$ we have
\begin{align}\label{Burkholder}
\E[|U_t-U_s|^p] \leq C_p|t-s|^{p/2}.
\end{align}
Hence, we may apply this result for the following terms: $b^{[1]},$ $b^{[2]},$ $b^{[1.1]},$ $b^{[1.2]},$ $b^{[2.1]}$ and $b^{[2.2]}.$
We expand and denote 
\begin{align}
 \Delta_n^{-1/4} V_n&= \frac{\Delta_n^{-1/4}}{\psi_2^n} \sum_{i=0}^{d_n-1}  \left( \widebar X_{t_{i k_n}} \right)^2+ 
\frac{2 \Delta_n^{-1/4}}{\psi_2^n} \sum_{i=0}^{d_n-1} \widebar X_{t_{i k_n}} \widebar \eps_{t_{i k_n}} \nonumber \\
 &+\frac{\Delta_n^{-1/4}}{\psi_2^n} \sum_{i=0}^{d_n-1}  \left[\widebar \eps_{t_{i k_n}}^2 -\frac{\psi_1^n \omega^2}{k_n} \right]
  +\frac{\Delta_n^{-1/4} \psi_1^n d_n}{\psi_2^n k_n }\left[\omega^2 -\frac{\Delta_n }{2} \sum_{i=1}^{ 1/\Delta_n } (\Delta Y_i^n)^2 \right] \nonumber \\
 &=: R_n^{(1)}+ R_n^{(2)}+R_n^{(3)}+ \Delta_n^{1/4} N_{n,6}+o_{\p}(\Delta_n^{1/4}). \label{Rn3}
 \end{align}
Let's look at the term $R_n^{[2]}$ first. Due to  
\begin{align*}
\widebar X_{t_{i k_n}}=&b_{t_{i k_n}}^{[1]} \widebar W_{t_{i k_n}}+\int_{t_{i k_n}}^{t_{(i+1)k_n}}  
\left[ (b_u^{[1]}-b_{t_{i k_n}}^{[1]}) dW(i,u) + g_n(u-t_{i k_n}) b_u^{[2]} du \right],
\end{align*}
we obtain 
\begin{align}
R_n^{(2)}=&\frac{2 \Delta_n^{-1/4}}{\psi_2^n} \sum_{i=0}^{d_n-1} \left \{  
 b_{t_{i k_n}}^{[1]} \widebar W_{t_{i k_n}}  \widebar \eps_{t_{i k_n}}
+ \int_{t_{i k_n}}^{t_{(i+1)k_n}}  
(b_u^{[1]}-b_{t_{i k_n}}^{[1]}) dW(i,u) \times \widebar \eps_{t_{i k_n}} \right \} \nonumber \\
&+\frac{2 \Delta_n^{-1/4}}{\psi_2^n} \sum_{i=0}^{d_n-1} \int_{t_{i k_n}}^{t_{(i+1)k_n}}  
 g_n(u-t_{i k_n}) b_u^{[2]} du  \times \widebar \eps_{t_{i k_n}} \nonumber \\
=& \frac{2 \Delta_n^{-1/4}}{\psi_2^n} \sum_{i=0}^{d_n-1}  b_{t_{i k_n}}^{[1]} \widebar W_{t_{i k_n}}  \widebar \eps_{t_{i k_n}} \nonumber \\
&+ \frac{2 \Delta_n^{-1/4}}{\psi_2^n}\sum_{i=0}^{d_n-1} \left(b_{t_{i k_n}}^{[1.1]} \int_{t_{i k_n}}^{t_{(i+1)k_n}}  
\int_{t_{i k_n}}^{u}  dW_s dW(i,u)+b_{t_{i k_n}}^{[2]} \psi_3^n k_n \Delta_n \right) \widebar\eps_{t_{ik_n}} \nonumber \\
&+o_{\p}(\Delta_n^{1/4}) \nonumber \\
=&: R_n^{(2.1)}+\Delta_n^{1/4}N_{n,5}+o_{\p}(\Delta_n^{1/4}). \label{Rn2}
\end{align}
In above computation, $o_{\p}(\Delta_n^{1/4})$-error terms were obtained by applying \eqref{Burkholder} to the processes $b^{[1.1]}$ and $b^{[2]}.$
Let's provide some details. \eqref{Burkholder} applied to the process $b^{[2]}$ implies
$$\E[|b_u^{[2]}-b_{t_{i k_n}}^{[2]}|^p]\leq C_p \Delta_n^{p/4}.$$
for each $0 \leq i \leq d_n-1$ and $t_{i k_n} \leq u \leq t_{(i+1) k_n}.$
Then, we obtain
\begin{align*}
&\int_{t_{i k_n}}^{t_{(i+1)k_n}}  
 g_n(u-t_{i k_n}) (b_u^{[2]}-b_{t_{i k_n}}^{[2]}) du  \times \widebar \eps_{t_{i k_n}}=O_{\p}(\Delta_n) \mbox{ and } \\
& \frac{2 \Delta_n^{-1/4}}{\psi_2^n} \sum_{i=0}^{d_n-1} \int_{t_{i k_n}}^{t_{(i+1)k_n}}  
 g_n(u-t_{i k_n}) (b_u^{[2]}-b_{t_{i k_n}}^{[2]}) du  \times \widebar \eps_{t_{i k_n}}=O_{\p}(\Delta_n^{1/2}),
\end{align*}
where we used the independece of $X$ and $\eps,$ and the i.i.d assumption on $\eps$ for the second result above.

Now, we pass to the expansion of $R_n^{(1)}.$ Analogous to $W(i,t)$ defined in \eqref{processW(i,t)}, we define a new process
\begin{align*}
X(i, t)&=\int_{t_{i k_n}}^{t} g_n(u-t_{i k_n}) b_u^{[1]} dW_u+\int_{t_{i k_n}}^{t} g_n(u-t_{i k_n}) b_u^{[2]} du.
\end{align*}
We remark that $$\widebar X_{t_{i k_n}}=X(i, t_{(i+1) k_n}).$$
Now, Ito's lemma yields 
\begin{align*}
(\widebar X_{t_{i k_n}})^2=& 2 \int_{t_{i k_n}}^{t_{(i+1)k_n}}  X(i, u) (b_u^{[1]} dW(i,u)+g_n(u-t_{i k_n}) b_u^{[2]} du)
 \\ &+\int_{t_{i k_n}}^{t_{(i+1)k_n}} (g_n(u-t_{i k_n}))^2 (b_u^{[1]})^2 du . 
\end{align*}
Therefore, we get
\begin{align*}
R_n^{(1)}= &\frac{2 \Delta_n^{-1/4}}{\psi_2^n} \sum_{i=0}^{d_n-1} \int_{t_{ik_n}}^{t_{(i+1)k_n}} X(i, u) b_u^{[1]} dW(i,u) \\
&+ \frac{2 \Delta_n^{-1/4}}{\psi_2^n} \sum_{i=0}^{d_n-1} \int_{t_{ik_n}}^{t_{(i+1)k_n}} X(i, u) g_n(u-t_{ik_n}) b_u^{[2]} du \\
&+\frac{\Delta_n^{-1/4}}{\psi_2^n} \sum_{i=0}^{d_n-1} \int_{t_{i k_n}}^{t_{(i+1) k_n}} (g_n(u-i k_n \Delta_n))^2 (b_u^{[1]})^2 du \\
=&: R_n^{(1.1)}+R_n^{(1.2)}+R_n^{(1.3)}.
\end{align*} 
Here, $R_n^{(1.1)}$ is decomposed as 
\begin{align}
R_n^{(1.1)}&= \frac{2 \Delta_n^{-1/4}}{\psi_2^n} \sum_{i=0}^{d_n-1} \int_{t_{i k_n}}^{t_{(i+1)k_n}} \int_{t_{ik_n}}^{u}  b_s^{[1]} dW(i,s)b_u^{[1]} dW(i,u) \nonumber\\
&+ \frac{2 \Delta_n^{-1/4}}{\psi_2^n} \sum_{i=0}^{d_n-1} \int_{t_{i k_n}}^{t_{(i+1)k_n}} \int_{t_{ik_n}}^{u} g_n(s-t_{ik_n}) b_s^{[2]} ds ~b_u^{[1]} dW(i,u) \nonumber \\
&= \frac{2 \Delta_n^{-1/4}}{\psi_2^n} \sum_{i=0}^{d_n-1} (b_{t_{i k_n}}^{[1]})^2 \int_{t_{i k_n}}^{t_{(i+1) k_n}} W(i,u)  dW(i,u) \nonumber \\
& +\frac{2 \Delta_n^{-1/4}}{\psi_2^n} \sum_{i=0}^{d_n-1} b_{t_{i k_n} }^{[1]} b_{t_{i k_n}}^{[1.1]} \int_{t_{i k_n}}^{t_{(i+1)k_n}} \int_{t_{i k_n}}^{u} (W_s-W_{t_{i k_n}}) dW(i,s) dW(i,u) \nonumber \\
&+ \frac{2 \Delta_n^{-1/4}}{\psi_2^n} \sum_{i=0}^{d_n-1} b_{t_{i k_n}}^{[1]} b_{t_{i k_n}}^{[2]} \int_{t_{i k_n}}^{t_{(i+1) k_n}} \int_{t_{i k_n}}^{u} g_n(s-t_{i k_n}) ds    dW(i,u) \nonumber \\
& +\frac{2 \Delta_n^{-1/4}}{\psi_2^n} \sum_{i=0}^{d_n-1} b_{t_{i k_n}}^{[1]} b_{t_{i k_n}}^{[1.1]} \int_{t_{i k_n}}^{t_{(i+1) k_n}} \int_{t_{i k_n}}^{u}  dW_s W(i,u) dW(i,u)
+ o_{\p}(\Delta_n^{1/4})  \nonumber \\
&=: R_n^{(1.1.1)}+R_n^{(1.1.2)}+R_n^{(1.1.3)}+R_n^{(1.1.4)}+ o_{\p}(\Delta_n^{1/4}). \label{Rn11}
\end{align}
Using \eqref{Rn3}, \eqref{Rn2} and \eqref{Rn11}, we immediately observe that 
\begin{align*}
R_n^{(1.1.1)}=&\frac{\Delta_n^{-1/4}}{\psi_2^n} \sum_{i=0}^{d_n-1} (b_{t_{i k_n}}^{[1]} \widebar W_{t_{i k_n}})^2-  \E[ (b_{t_{i k_n}}^{[1]}\widebar W_{t_{i k_n}})^2 | \F_{t_{i k_n}}] \mbox{ and } \\
M_n=&R_n^{(1.1.1)}+R_n^{(2.1)}+R_n^{(3)}.
\end{align*} 
Regarding the term $R_n^{(1.2)},$ we proceed similarly and obtain 
\begin{align}
&R_n^{(1.2)}= \frac{2 \Delta_n^{-1/4}}{\psi_2^n} \sum_{i=0}^{d_n-1} b_{t_{i k_n}}^{[1]} \int_{t_{i k_n}}^{t_{(i+1)k_n}} W(i,u)g_n(u-t_{i k_n}) b_u^{[2]} du \nonumber \\
&+ \frac{2 \Delta_n^{-1/4}}{\psi_2^n} \sum_{i=0}^{d_n-1} (b_{t_{i k_n}}^{[2]})^2 \int_{t_{i k_n}}^{t_{(i+1)k_n}} \int_{t_{i k_n}}^{u} g_n(s-t_{i k_n})ds~ g_n(u-t_{i k_n})  du \nonumber 
%\\&
+o_{\p}(\Delta_n^{1/4}) \nonumber\\
&=\frac{2 \Delta_n^{-1/4}}{\psi_2^n} \sum_{i=0}^{d_n-1} b_{t_{i k_n}}^{[1]} b_{t_{i k_n}}^{[2]} \int_{t_{i k_n}}^{t_{(i+1)k_n}}W(i,u)g_n(u-t_{i k_n})du \nonumber \\ 
&+\frac{2 \Delta_n^{-1/4}}{\psi_2^n} \sum_{i=0}^{d_n-1} b_{t_{i k_n}}^{[1]} b_{t_{i k_n}}^{[2.1]} \int_{t_{i k_n}}^{t_{(i+1)k_n}}
\int_{t_{i k_n}}^{u} dW_s W(i,u)g_n(u-t_{i k_n})du \nonumber \\ 
&+ \frac{\Delta_n^{-1/4}}{\psi_2^n} \sum_{i=0}^{d_n-1} (b_{t_{i k_n}}^{[2]})^2 \left( \int_{t_{i k_n}}^{t_{(i+1)k_n}} g_n(u-t_{i k_n})du \right)^2+o_{\p}(\Delta_n^{1/4}) \nonumber \\
&=: R_n^{(1.2.1)}+\Delta_n^{1/4} N_{n,2}+\Delta_n^{1/4} N_{n,3}+o_{\p}(\Delta_n^{1/4}). \label{Rn12}
\end{align}
In view of \eqref{Rn11} and \eqref{Rn12}, we note that Ito's product rule yields 
\begin{align}
 R_n^{(1.1.3)}+R_n^{(1.2.1)}=&\frac{2 \Delta_n^{-1/4}}{\psi_2^n} \sum_{i=0}^{d_n-1} b^{[1]}_{t _{i k_n}}  
 b_{t_{i k_n}}^{[2]} \widebar{W} _{t _{i k_n}}
 \int_{t_{i k_n}}^{t_{(i+1)k_n}}g_n(u-t_{i k_n})du \nonumber \\
=& \frac{2 \psi_3^n k_n \Delta_n^{3/4}}{\psi_2^n} \sum_{i=0}^{d_n-1} b^{[1]}_{t _{i k_n}}  
 b_{t_{i k_n}}^{[2]} \widebar{W} _{t _{i k_n}}
  =: \Delta_n^{1/4} N_{n,1}^{(1)}. \label{Nn11}
\end{align}
Finally, we look at $R_n^{(1.3)}$ and $V.$ For both terms we use
 $$(b_u^{[1]})^2=(b_{t_{i k_n}}^{[1]})^2+ 2 \int_{t_{i k_n}}^u  b_s^{[1]} \left( b_s^{[1.1]}dW_s+b_s^{[1.2]} ds \right)+ \int_{t_{i k_n}}^{u} (b_s^{[1.1]})^2 ds.$$
Then, recalling the definition of $Z_n^{\star}$ in \eqref{znstar} and the estimate \eqref{znnegl}, 
we get
\begin{align}
& R_n^{(1.3)}-\Delta_n^{-1/4} \int_{0}^{d_n k_n \Delta_n} (b_u^{[1.1]})^2 du \nonumber \\
&= \frac{2 \Delta_n^{-1/4}}{\psi_2^n} \sum_{i=0}^{d_n-1} b_{t_{i k_n} }^{[1]} b_{t_{i k_n}}^{[1.1]}
\int_{t_{i k_n}}^{t_{(i+1) k_n}} (g_n^2(u-t_{i k_n})-\psi_2^n) (W_u -W_{t_{i k_n}})  du \nonumber \\
&+ \frac{\Delta_n^{-1/4}}{\psi_2^n}\sum_{i=0}^{d_n-1}  \left[2 b_{t_{i k_n} }^{[1]} b_{t_{i k_n}}^{[1.2]}+(b_{t_{i k_n}}^{[1.1]})^2 \right] 
\int_{t_{i k_n}}^{t_{(i+1) k_n}} ((g_n(u-t_{i k_n}))^2-\psi_2^n)
u du \nonumber \\
&=: \Delta_n^{1/4}(N_{n,1}^{(2)}+N_{n,4}). \label{Nn12}
\end{align}
Using \eqref{Rn11}, \eqref{Nn11}, \eqref{Nn12} and Ito's formula, we obtain
$$R_n^{(1.1.2)}+R_n^{(1.1.4)}+\Delta_n^{1/4}\left(N_{n,1}^{(1)}+N_{n,1}^{(2)} \right)=\Delta_n^{1/4} N_{n,1}.$$
This finishes the proof of Proposition \ref{StochExpan}.
\end{proof}

\begin{proof}[Proof of Theorem \ref{CLT2}]
We write 
\begin{align*}
M_n&=\sum_{i=0}^{d_n-1} \chi_{i,1}^n, ~~~
N_n=\sum_{k=2}^4 N_{n,k}+\sum_{i=0}^{d_n-1} \chi_{i,2}^n, \\
\widehat C_n&=K_n  + \sum_{i=0}^{d_n-1} \chi_{i,3}^n, ~~~
\widehat F_n=L_n  +\sum_{i=0}^{d_n-1} \chi_{i,4}^n, 
\end{align*}
where
\begin{align*}
\chi_{i,1}^n&=\frac{\Delta_n^{-1/4}}{\psi_2^n}\left(\alpha_{t_{ik_n}}^2-\E[\alpha_{t_{ik_n}}^2| \mathcal{F}_{t_{ik_n}}] \right), \\
\chi_{i,2}^n&=\frac{2 \Delta_n^{-1/2}}{\psi_2^n}  b_{t_{i k_n}}^{[1]} \int_{t_{i k_n}}^{t_{(i+1) k_n}} \nu_i^n(u) dW_u+  \\
& \frac{2 \Delta_n^{-1/2}}{\psi_2^n} \widebar \eps_{t_{i k_n}}  
\left( \psi_3^n k_n \Delta_n b_{t_{i k_n}}^{[2]}   +b_{t_{i k_n}}^{[1.1]} \int_{t_{i k_n}}^{t_{(i+1)k_n}}  \int_{t_{i k_n}}^{u}  dW_s dW(i,u)  \right),\\
&+ \frac{\Delta_n^{-1/2} \psi_1^n}{2 \psi_2^n (k_n)^2} \sum_{j=1}^{ k_n } 
\left[2 \omega^2-(\Delta \eps_{ i k_n+j} ^n)^2 \right],\\ 
\chi_{i,3}^n&=\frac{4 \Delta_n^{-3/4}}{(\psi_2^n)^2}  \int_{t_{i k_n-1}}^{t_{(i+1) k_n-1}}
\left(\alpha^2(i, u )- \E[\alpha^2(i, u )| \mathcal{F}_{t_{ik_n}}] \right)[d \alpha(i, u)]^2,  \\
\chi_{i,4}^n&= \frac{2 \Delta_n^{-3/4}}{3 (\psi_2^n)^2} \left(\alpha_{t_{ik_n}}^4-\E[\alpha_{t_{ik_n}}^4| \mathcal{F}_{t_{ik_n}}] \right)
,\\
K_n&=\Delta_n^{-1/4} \left(\frac{4 \Delta_n^{-1/2}}{(\psi_2^n)^2} \sum_{i=0}^{d_n-1} \int_{t_{i k_n-1}}^{t_{(i+1) k_n-1}}\E[\alpha^2(i, u )| \mathcal{F}_{t_{ik_n}}][d \alpha(i, u)]^2-C \right), \\
L_n&=\frac{2 \Delta_n^{-3/4}}{3 (\psi_2^n)^2} \sum_{i=0}^{d_n-1} \left((\widebar Y_{t_{i k_n}})^4-\alpha_{t_{ik_n}}^4\right)\\
&+
\Delta_n^{-1/4} 
\left(\frac{2 \Delta_n^{-1/2}}{3 (\psi_2^n)^2} \sum_{i=0}^{d_n-1} \E[\alpha_{t_{ik_n}}^4| \mathcal{F}_{t_{ik_n}}]- C \right),
\end{align*}
and the quantity $[d \alpha(i, u)]^2$ has been introduced in \eqref{dalpha}. 
We observe that $K_n \cp 0$ and $L_n \cp 0.$ We recall that $\alpha_{t_{ik_n}}$, conditionally on $\mathcal{F}_{t_{i k_n}},$ is distributed as
$$N\left(0, \psi_2^n k_n \Delta_n (b_{t_{i k_n}}^{[1]})^2  + \frac{\psi_1^n \omega^2}{k_n} \right).$$  
Then, for $\chi_{i}^n=(\chi_{i,1}^n, \chi_{i,2}^n, \chi_{i,3}^n, \chi_{i,4}^n)$ we obtain
\begin{align*}
&\sum_{i=0}^{d_n-1} \E[\chi_{i,k}^n \chi_{i,l}^n| \mathcal{F}_{t_{ik_n}}] \cp \int_0^1 \Sigma_s^{kl} ds, \\
&\sum_{i=0}^{d_n-1} \E[\chi_{i,k}^n (W_{t_{(i+1) k_n}}-W_{t_{i k_n}})| \mathcal{F}_{t_{ik_n}}] \cp \int_0^1 \mu_s^k ds \\
&\sum_{i=0}^{d_n-1} \E[\chi_{i,k}^n (B_{t_{(i+1) k_n}}-B_{t_{i k_n}}) | \mathcal{F}_{t_{ik_n}}] \cp 0
\end{align*}
for $1 \leq k, l \leq 4$, where the stochastic processes $\Sigma^{kl}$ and $\mu^k$ are introduced in Theorem \ref{CLT2}. 
For any $\delta>0$ and $1 \leq k \leq 4,$ we observe that
$$\sum_{i=0}^{d_n-1} \E[|\chi_{i,k}^n|^2 \mathbbm{1}_{ \{ |\chi_{i,k}^n| > \delta \}} | \mathcal{F}_{t_{ik_n}} ]
\leq \delta^{-2} \sum_{i=0}^{d_n-1} \E[|\chi_{i,k}^n|^4 | \mathcal{F}_{t_{ik_n}} ]
\leq C_{\epsilon} k_n \Delta_n \to 0.$$
Now, let $Q$ be a bounded continuous martingale with $\langle W, Q \rangle 
=\langle B, Q \rangle=0.$
A standard argument (see e.g. \cite{JLMP09}) shows that $$\E[\chi_{i,k}^n (Q_{t_{(i+1) k_n}}-Q_{t_{i k_n}}) | \mathcal{F}_{t_{ik_n}}]=0.$$
This implies $$\sum_{i=0}^{d_n-1} \E[\chi_{i,k}^n (Q_{t_{(i+1) k_n}}-Q_{t_{i k_n}}) | \mathcal{F}_{t_{ik_n}}] \to 0.$$
Then, we are done using Theorem IX.7.28 of \cite{JASH03}.
\end{proof}

\subsection{Nondegeneracy of $C$}
For
$$C=2 \theta \int_0^1 \left((b_u^{[1]})^2+ \frac{\omega^2 \psi_1}{\theta^2 \psi_2} \right)^2du,$$
we get
$$D_r^{(1)} C=8 \theta \int_r^1 \left((b_u^{[1]})^2+ \frac{\omega^2 \psi_1}{\theta^2 \psi_2} \right) b_u^{[1]} (b_u^{[1]})'D_r^{(1)} X_u du$$
and 
$D_r^{(2)} C=0$ for $r\in[0,1]$.
We write 
$$
\sigma_{22}(t)=\int_0^t 
\left(
8 \theta \int_r^1 \left((b_u^{[1]})^2+ \frac{\omega^2 \psi_1}{\theta^2 \psi_2} \right) b_u^{[1]} (b_u^{[1]})' D_r^{(1)} X_u du 
\right)^2dr.
$$
Then $C$'s Malliavin covariance $\sigma_{C} = \sigma_{22}(1)$, i.e. $\|DC\|_{\mathbb H} ^2=\sigma_{22}(1)$. 

We will work with the two-dimensional stochastic process 
$\bar{X}_t=(X^{(1)}_t,X^{(2)}_t)$ defined by the stochastic integral equation
\beas
\bar{X}_t
&=& 
\bar{X}_0+
\int_0^tV_1(\bar{X}_s)\circ dW_s+\int_0^tV_0(\bar{X}_s)ds
\eeas
for $t\in[0,1]$, where $\circ$ denotes the Stratonovich integral. 
The coefficients are given by 
\beas 
V_1(x)\>=\>\left[\begin{array}{c}b^{[1]}(x^{(1)})\\0\end{array}\right],\qquad
V_0(x)\>=\>\left[\begin{array}{c}\tilde{b}^{[2]}(x^{(1)})\\ a(x^{(1)})\end{array}\right]
\eeas
for $x=(x^{(1)},x^{(2)})$, 
$\tilde{b}^{[2]}=b^{[2]}-2^{-1}b^{[1]}(b^{[1]})'$ and  
\beas 
a(x^{(1)}) &=& 2\theta \left((b^{[1]}(x^{(1)}))^2+ \frac{\omega^2 \psi_1}{\theta^2 \psi_2} \right)^2
\eeas
already defined.  Condition [V] implies the H\"ormander condition
\beas 
\text{Lie}[V_0;V_1](x^{(1)},0)\>=\>\bbR^2\qquad(\forall x^{(1)}\in\text{supp}\call\{X_0\})
\eeas
where $\mbox{Lie}[V_0;V_1],$ the Lie algebra generated by $V_1$ and $V_0,$ is defined in the following way.
The Lie bracket of $V$ and $W$ is given by
$$[V, W](x)=\mathcal{D} V(x) W(x)-\mathcal{D} W(x) V(x),$$
where $\mathcal{D} V(x)$ is the derivative of $V$ at $x.$ Then, we define the Lie algebra by 
$\mbox{Lie}[V_0;V_1]=\text{span}\left(\bigcup_{j=0}^\infty\Sigma_j\right)$, where 
 $\Sigma_0=\{V_1\}$ and $\Sigma_j=\{[V,V_i];\> V\in\Sigma_{j-1},\>i=0,1\}$ ($j\geq1$). 
 %with the Lie bracket $[\cdot,\cdot]$. $\mbox{Lie}[V_0;V_1](x)$ is $\mbox{Lie}[V_0;V_1]$ evaluated at $x$.  
It is then possible to deduce that 
\bea\label{20150817-1}
\sigma_{22}(t)^{-1}&\in&\bigcap_{p>1}L^p
\eea
for every $t\in(0,1]$. 
For details, we refer the reader to \cite{POYO13,yoshida2012asymptotic,YOSH13}.

%%%%%%%%%%%%%%%%%%%%%%%
\subsection{Setting $s_n$ and a local nondegeneracy of $(M^n_t,C)$}
The Malliavin covariance matrix of $(M^n_t,C)$ is denoted by 
\beas 
\sigma_{(M^n_t,C)}
&=& 
\left[ \begin{array}{cc} \sigma_{11}(n,t)&\sigma_{12}(n,t)\\
\sigma_{21}(n,t)&\sigma_{22}(1) \end{array}
\right].
\eeas
Recall that 
\beas 
\alpha(i,t) &=& b^{[1]}_{t_{ik_n}}W(i,t)+\ep(i,t),
\\
W(i,t)&=& \int_{t_{ik_n}\wedge t}^{t_{(i+1)k_n}\wedge t}g_n(u-t_{ik_n})dW_u,
\\
\ep(i,t) &=& \omega\Delta_n^{-1/2}
\int_{t_{ik_n-1}\wedge t}^{t_{(i+1)k_n-1}\wedge t}-h_n(u-t_{ik_n})dB_u.
\eeas
Denoting by $\langle U \rangle$ the quadratic variation process of $U$, we conclude that  
\beas 
\langle \alpha(i,\cdot)\rangle_t 
&=& 
(b^{[1]}_{t_{ik_n}})^2\int_{t_{ik_n}\wedge t}^{t_{(i+1)k_n}\wedge t}g_n(u-t_{ik_n})^2du
\\&&
+\omega^2\Delta_n^{-1}\int_{t_{ik_n-1}\wedge t}^{t_{(i+1)k_n-1}\wedge t}h_n(u-t_{ik_n})^2du. 
\eeas
and 
\beas 
D_r^{(1)}W(i,t) 
&=& 
g_n(r-t_{ik_n})\mathbbm{1}_{(t_{ik_n}\wedge t,t_{(i+1)k_n}\wedge t]}(r)
\\&=&
g_n(r-t_{ik_n})\mathbbm{1}_{(t_{ik_n},t_{(i+1)k_n}]}(r)\mathbbm{1}_{\{r\leq t\}}. 
\eeas
Obviously, 
\beas 
\E[\alpha_{t_{ik_n}}^2|\calf_{t_{ik_n}}]
&=& 
(b^{[1]}_{t_{ik_n}})^2\psi^n_2k_n\Delta_n+\omega^2\Delta_n^{-1}
\int_{t_{ik_n-1}}^{t_{(i+1)k_n-1}}h_n(u-t_{ik_n})^2du.
\eeas
Since 
\beas 
M_t^n 
&=&\frac{2 \Delta_n^{-1/4}}{\psi_2^n} \sum_{i=0}^{d_n-1} \int_{t_{i k_n} \wedge t}^{t_{(i+1) k_n} \wedge t} \alpha(i, u)  d \alpha(i, u).  
\\&=&
\frac{ \Delta_n^{-1/4}}{\psi_2^n} \sum_{i=0}^{d_n-1} 
\big\{ \alpha(i,t)^2-\langle \alpha(i,\cdot)\rangle_t\big\},
\eeas
it follows that 
\begin{align*}
& D_r^{(1)}  M_{t}^n=\frac{\Delta_n^{-1/4}}{\psi_2^n} \\
&\times \sum_{i=0}^{d_n-1} 
\Bigg\{
2 \alpha(i,t) \Big((b^{[1]})'_{t_{i k_n}} D_r^{(1)} X_{t_{i k_n}} W(i,t)\mathbbm{1}_{ (0, t_{ik_n}]}(r) 
+ b_{t_{i k_n}}^{[1]} D_r^{(1)} {W(i,t)} \Big)\\&- 2 b_{t_{i k_n}}^{[1]} (b^{[1]})'_{t_{i k_n}} D_r^{(1)} X_{t_{i k_n}} 
\int_{t_{ik_n}\wedge t}^{t_{(i+1)k_n}\wedge t}g_n(u-t_{ik_n})^2du
\mathbbm{1}_{ (0, t_{ik_n}]}(r)
\Bigg\} \\
&=
\frac{\Delta_n^{-1/4}}{\psi_2^n} \sum_{i=0}^{d_n-1} 
\Bigg[
\sum_{l=i+1}^{d_n-1} %\sum_{l=i}^{p-1} 
2(b^{[1]})'_{t_{l k_n}} D_r^{(1)} X_{t_{l k_n}} 
\bigg\{\alpha(l,t)W(l,t) \\
&- b_{t_{l k_n}}^{[1]}  \int_{t_{lk_n}\wedge t}^{t_{(l+1)k_n}\wedge t}g_n(u-t_{lk_n})^2du\bigg\}  
+ 2 \alpha(i,t) b_{t_{i k_n}}^{[1]} G_i(r) \mathbbm{1}_{\{r\leq t\}}
\Bigg] \mathbbm{1}_{I_i}(r).   
\end{align*}
and 
\begin{align*}
D_r^{(2)} M_{t}^n&=
\frac{-\Delta_n^{-1/4}}{\psi_2^n} \sum_{i=0}^{d_n-1} 
2 \alpha(i,t) \omega\Delta_n^{-1/2} h_n(r-t_{ik_n})
\mathbbm{1}_{(t_{ik_n-1}\wedge t,t_{(i+1)k_n-1}\wedge t]}(r)
\\&=
\frac{\Delta_n^{-1/4}}{\psi_2^n} \sum_{i=0}^{d_n-1} 
2 \alpha(i,t) \omega\Delta_n^{-1/2} H_i(r)
\mathbbm{1}_{\{r\leq t\}}
\times\mathbbm{1}_{(t_{ik_n-1},t_{(i+1)k_n-1}]}(r).
\end{align*}
where $I_i=(t_{ik_n}, t_{(i+1)k_n}],$  $I_{i,j}=(t_{ik_n+j-1}, t_{ik_n+j}],$ and we use the notations 
$G_i(r)=\sum_{j=1}^{k_n-1} g(j/k_n) \mathbbm{1}_{ I_{i,j}}(r)$ and $H_i(r)=\sum_{j=0}^{k_n-1} -h(j/k_n) \mathbbm{1}_{ I_{i,j}}(r).$
Hence, the Malliavin covariance matrix of $M_t^n$ is expressed by
\begin{align*} 
&\sigma_{M_t^n}
\equiv
\sigma_{11}(n,t)
\\&=
\frac{1}{(\psi_2^n)^2} \sum_{i=0}^{d_n-1} \int_{t_{ik_n}}^{t_{(i+1)k_n}}
\Bigg[
2 \Delta_n^{-1/4}\alpha(i,t) b_{t_{i k_n}}^{[1]} G_i(r) \mathbbm{1}_{\{r\leq t\}}+
\Delta_n^{1/4}\sum_{l=i+1}^{d_n-1}
2(b_{t_{l k_n}}^{[1]})'
\\&
 \times D_r X_{t_{l k_n}} 
\Delta_n^{-1/2}\big[ \alpha(l,t)W(l,t)-  b_{t_{l k_n}}^{[1]} 
\int_{t_{lk_n}\wedge t}^{t_{(l+1)k_n}\wedge t}g_n(u-t_{lk_n})^2du
\big]  
\Bigg]^2\>dr
\\&
+ 
\frac{1}{(\psi_2^n)^2} \sum_{i=0}^{d_n-1} 
\int_{t_{ik_n-1}\wedge t}^{t_{(i+1)k_n-1}\wedge t}
[2 \Delta_n^{-1/4}\alpha(i,t) \omega \Delta_n^{-1/2} 
H_i(r)]^2   %{\colorr{\Bigg\}}} 
%\mathbbm{1}_{\{r\leq t\}}
\>dr.
\end{align*}
%}}
%
\begin{en-text}
\beas 
\sigma_{M_{t_{(p+1)k_n}}^n}
&=&
\frac{\Delta_n^{-1/2}}{(\psi_2^n)^2} \sum_{i=0}^{p} \int_{{\colorr{t_{ik_n}}}}^{\colorr{t_{(i+1)k_n}}}
{\colorr{\Bigg\{}}
\Bigg[
{\colorr{\sum_{l=i+1}^p%\sum_{l=i}^{p-1} 
}}
(b_{t_{l k_n}}^{[1]})' D_r X_{t_{l k_n}} [2 \alpha_{t_{l k_n}} \widebar W_{t_{l k_n}} - 2 b_{t_{l k_n}}^{[1]}  \psi_2^n k_n \Delta_n)]  
\\&&
+ 2 \alpha_{t_{i k_n}} b_{t_{i k_n}}^{[1]} G_i(r) \Bigg]^2+ [2 \alpha_{t_{i k_n}} \omega \Delta_n^{-1/2} 
H_i(r)]^2   {\colorr{\Bigg\}}} dr
\eeas
\end{en-text}
The cross Malliavin covariance $\langle M^n_t,C  \rangle_{\mathbb H}$ of $(M^n_t,C)$ is given by 
\begin{align*}
&\sigma_{12}(t)= 
\frac{1}{\psi_2^n} \sum_{i=0}^{d_n-1} \int_{t_{ik_n}}^{t_{(i+1)k_n}}
\Bigg\{
\Bigg[
2 \Delta_n^{-1/4}\alpha(i,t) b_{t_{i k_n}}^{[1]} G_i(r)\mathbbm{1}_{\{r\leq t\}}+2 \Delta_n^{1/4} \times
\\&
\sum_{l=i+1}^p
(b^{[1]})'_{t_{l k_n}} D_r^{(1)} X_{t_{l k_n}} 
\Delta_n^{-1/2}\big[\alpha(l,t)W(l,t) - b_{t_{l k_n}}^{[1]} 
\int_{t_{lk_n}\wedge t}^{t_{(l+1)k_n}\wedge t}g_n^2(u-t_{lk_n})du
\big]  
\Bigg]
\\&
\times\Bigg[
8 \theta \int_r^1 \left((b_u^{[1]})^2+ \frac{\omega^2 \psi_1}{\theta^2 \psi_2} \right) b_u^{[1]} (b_u^{[1]})' D_r^{(1)}  X_u du 
\Bigg]
\Bigg\} dr.
\end{align*}
Now, let \beas
\sigma(n,t)
&=&
\begin{bmatrix}  \sigma_{11}(n,t) & \sigma_{12}(n,t) \\
\sigma_{12}(n,t) & \sigma_{22}(t) \end{bmatrix}.
\eeas
Then, we modify $\sigma(n,t)$ for $t=t_{(p+1)k_n}$ by
\beas
\tilde{ \sigma}(n,t)
&=&
\begin{bmatrix}  \tilde{ \sigma}_{11}(n,t) & \tilde{\sigma}_{12}(n,t) \\
\tilde{ \sigma}_{12}(n,t) & \sigma_{22}(t) \end{bmatrix}
\qquad(t=t_{(p+1)k_n})
\eeas
with 
\begin{align*}
\tilde{ \sigma}_{11}(n,t)
&=
\frac{1}{(\psi_2^n)^2} \sum_{i=0}^{p} \int_{t_{ik_n}}^{t_{(i+1)k_n}}
\Big[
2 \Delta_n^{-1/4}\alpha_{t_{ik_n}} b_{t_{i k_n}}^{[1]} G_i(r) 
\Big]^2\>dr
\\
+
\frac{1}{(\psi_2^n)^2} & \sum_{i=0}^{p} 
\int_{t_{ik_n-1}}^{t_{(i+1)k_n-1}}
[2 \Delta_n^{-1/4}\alpha_{t_{ik_n}} \omega \Delta_n^{-1/2} 
H_i(r)]^2   \>dr+\frac{1}{(\psi_2^n)^2} \sum_{i=0}^{p} \int_{t_{ik_n}}^{t_{(i+1)k_n}}
\\
\times \Bigg[
\Delta_n^{1/4} &\sum_{l=i+1}^{p} 2(b^{[1]})'_{t_{l k_n}} D_r^{(1)}  X_{t_{l k_n}} 
\Delta_n^{-1/2}\big[ \alpha_{t_{lk_n}}\overline{W}_{t_{lk_n}}-  b_{t_{l k_n}}^{[1]} 
\psi^n_2k_n\Delta_n
\big]  
\Bigg]^2\>dr.
\end{align*}
and 
\beas 
\tilde{\sigma}_{12}(n,t) 
&=& 
\frac{1}{\psi_2^n} \sum_{i=0}^{p} \int_{t_{ik_n}}^{t_{(i+1)k_n}}
\Bigg\{
\Bigg[
\Delta_n^{1/4}\sum_{l=i+1}^p 2(b^{[1]})'_{t_{l k_n}} D_r^{(1)}  X_{t_{l k_n}} 
\\&&\times
\Delta_n^{-1/2}\big[\alpha_{t_{lk_n}}\overline{W}_{t_{lk_n}} - b_{t_{l k_n}}^{[1]} 
\int_{t_{lk_n}\wedge t}^{t_{(l+1)k_n}\wedge t}g_n(u-t_{lk_n})^2du
\big]  
\Bigg]
\\&&
\times\Bigg[
8 \theta \int_r^1 \left((b_u^{[1]})^2+ \frac{\omega^2 \psi_1}{\theta^2 \psi_2} \right) b_u^{[1]} (b_u^{[1]})' D_r^{(1)}  X_u du 
\Bigg]
\Bigg\} dr,
\eeas
evaluated at $t=t_{(p+1)k_n}$, 
respectively. 
Let 
\beas 
J_{n,l}^{(1)} \>=\> \Delta_n^{-1/4}\overline{W}_\tlk 
&=& 
\Delta_n^{-1/4}\int_\tlk^\tlkp g_n(s-lk_n\Delta_n)dW_s
\eeas
and 
\beas
J_{n,l}^{(2)}  \>=\>\Delta_n^{-1/4}\overline{\ep}_\tlk 
&=& 
-\omega\Delta_n^{-3/4}\int_\tlk^\tlkp h_n(s-lk_n\Delta_n)dB_s.
\eeas
By definition, 
\beas 
\alpha_\tlk &=& 
\Delta_n^{1/4}\big(b^{[1]}_\tlk J^{(1)}_{n,l}+J^{(2)}_{n,l}\big)
\eeas
and 
\beas 
\alpha_\tlk\overline{W}_\tlk-b^{[1]}_\tlk\psi^n_2k_n\Delta_n
&=&
b^{[1]}_\tlk\big(\overline{W}_\tlk^2-\psi^n_2k_n\Delta_n\big)
+\overline{\ep}_\tlk\overline{W}_\tlk
\\&=&
\Delta_n^{1/2}J_{n,l}^{(3)}+\Delta_n^{1/2}J_{n,l}^{(4)}+\Delta_n^{1/2}J_{n,l}^{(5)}
\eeas
where 
\beas
J_{n,l}^{(3)}
&=&
\Delta_n^{-1/2}
b^{[1]}_\tlk\int_\tlk^\tlkp dW_s 2g_n(s-lk_n\Delta_n)\int_\tlk^sdW_ug_n(u-lk_n\Delta_n)dW_u
\\
J_{n,l}^{(4)}
&=&
%k_n\overline{\ep}_\tlk \int_\tlk^\tlkp g_n(s-lk_n\Delta_n)dW_s
%\\&=&
%-\omega\Delta_n^{-1/2}k_n\int_\tlk^\tlkp
%\int_\tlk^\tlkp h_n(u-\tlk)dB_u  g_n(s-lk_n\Delta_n)dW_s
%\\&=&
\Delta_n^{-1/2}
\int_\tlk^\tlkp  dW_sg_n(s-lk_n\Delta_n)
\int_\tlk^s dB_u (-\omega)\Delta_n^{-1/2}h_n(u-\tlk)
\\
J_{n,l}^{(5)}
&=&
\Delta_n^{-1/2}
\int_\tlk^\tlkp  dB_s (-\omega)\Delta_n^{-1/2}h_n(s-\tlk)
\int_\tlk^s dW_u g_n(u-lk_n\Delta_n)
\eeas
We claim that for every $q>1$, 
\bea\label{270813-1} 
\sup_{p}
\Big\|
\Delta_n^{1/2}\sum_{i=0}^p
\calc(i,p)
\Big\|_q
&=&
O(\Delta_n^{1/4})
\eea
as $n\to\infty$ 
for 
\begin{align*}
&\calc(i,p)=
\Delta_n^{-1/4} 
\alpha_{t_{i k_n}} b_{t_{i k_n}}^{[1]} %G_i(r)
\Delta_n^{1/4} 
\\&\times
\sum_{l=i+1}^p (b_{t_{l k_n}}^{[1]})' 
\left(\Delta_n^{-1/2}\int_{I_i}D_r^{(1)}  X_{t_{l k_n}} dr\right)
\Delta_n^{-1/2} 
[ \alpha_{t_{l k_n}} \widebar W_{t_{l k_n}} -  b_{t_{l k_n}}^{[1]}  \psi_2^n k_n \Delta_n)] .
\end{align*}
We have 
\beas 
\calc(i,p)
&=&
\bigg((b_{t_{i k_n}}^{[1]})^2J^{(1)}_{n,l}+b_{t_{i k_n}}^{[1]}J^{(2)}_{n,l}\bigg) %G_i(r)
\\& &\times
\Delta_n^{1/4} 
\sum_{l=i+1}^p (b_{t_{l k_n}}^{[1]})' 
\left(\Delta_n^{-1/2}\int_{I_i}D_r^{(1)}  X_{t_{l k_n}} dr\right)
\big(J^{(3)}_{n,l}+J^{(4)}_{n,l}+J^{(5)}_{n,l}\big) .
\eeas
Then the estimate (\ref{270813-1}) follows from Lemma 5 of \cite{YOSH13}. 
By using (\ref{270813-1}), we see 
\begin{align*} 
\sigma_{M_{t_{(p+1)k_n}}^n}&=
\frac{1}{(\psi_2^n)^2} \sum_{i=0}^{p} \int_{t_{ik_n}}^{t_{(i+1)k_n}}
\Bigg\{
\Big[
2 \Delta_n^{-1/4}\alpha_{t_{i k_n}} b_{t_{i k_n}}^{[1]} G_i(r) \Big]^2
\\
+&\Bigg[
\Delta_n^{1/4}\sum_{l=i+1}^p (b_{t_{l k_n}}^{[1]})' D_r^{(1)}  X_{t_{l k_n}} 
\Delta_n^{-1/2}[2 \alpha_{t_{l k_n}} \widebar W_{t_{l k_n}} - 2 b_{t_{l k_n}}^{[1]}  \psi_2^n k_n \Delta_n)]  
\Bigg]^2
\\&+ \Big[2 \Delta_n^{-1/4}\alpha_{t_{i k_n}} \omega \Delta_n^{-1/2} H_i(r)\Big]^2  \Bigg\} dr +O_{L^q}(\Delta_n^{1/4})
\\&=\hat{\sigma}_{11}(n,t_{(p+1)k_n}) +O_{L^q}(\Delta_n^{1/4})
\end{align*}
as $n\to\infty$ uniformly in $p$, for every $q>1$. 
That is, 
\beas 
\sup_{(t,p):\>t=t_{(p+1)k_n},\atop
\>p=0,...,d_n-1}
\big\|\sigma_{11}(n,t)-\tilde{\sigma}_{11}(n,t)\big\|_q &=& O(\Delta_n^{1/4})
\eeas
for every $q>1$. 
In a similar fashion, we also obtain 
\beas 
\sup_{(t,p):\>t=t_{(p+1)k_n},\atop
\>p=0,...,d_n-1}
\big\|\sigma_{12}(n,t)-\tilde{\sigma}_{12}(n,t)\big\|_q &=& O(\Delta_n^{1/4})
\eeas
as $n\to\infty$ 
for every $q>1$. In this way, we obtain the following result.  
\begin{lemma}\label{270815-1}
\beas
\sup_{(t,p):\>t=t_{(p+1)k_n},\atop
\>p=0,...,d_n-1}
\big\|\sigma(n,t)-\tilde{\sigma}(n,t)\big\|_q &=& O(\Delta_n^{1/4})
\eeas
as $n\to\infty$ 
for every $q>1$. 
\end{lemma}

\begin{en-text}
\koko para Let 
\beas 
\hat{\sigma}_{11}(n,t)
&=&
\frac{1}{(\psi_2^n)^2} \sum_{i=0}^{p} 
\int_{{\colorr{t_{ik_n}}}}^{\colorr{t_{(i+1)k_n}}}
\Bigg[
\sum_{l=i}^{p-1}4 (b^{[1]})'_{t_{l k_n}} D_r X_{t_{l k_n}} 
\Delta_n^{-1/4}
\big[\alpha_{t_{l k_n}} \widebar W_{t_{l k_n}} - b_{t_{l k_n}}^{[1]}  \psi_2^n k_n \Delta_n
\big] \Bigg]^2 dr  
\\&&
+ 
\frac{1}{(\psi_2^n)^2} \sum_{i=0}^{p} 
4\Delta_n^{-1/2}
 \big([b_{t_{i k_n}}^{[1]} \widebar W_{t_{i k_n}}]^2+[\widebar \eps_{t_{i k_n}}]^2\big)
 \bigg( (b_{t_{i k_n}}^{[1]})^2 \psi_2^n k_n \Delta_n +\frac{\omega^2 \psi_1^n}{k_n}\bigg)
\eeas
for $t=t_{(p+1)k_n}$. 
Then using orthogonality between $\overline{W}_{t_{ik_n}}$ and $\overline{\ep}_{t_{ik_n}}$,  
we obtain  
\beas 
\sup_{(t,p):\>t=t_{(p+1)k_n},\atop
\>p=0,...,d_n-1}
\big\|\sigma_{11}(n,t)-\hat{\sigma}_{11}(n,t)\big\|_q &=& O(\Delta_n^{1/4})
\eeas
for every $q>1$. 
\end{en-text}

Define $m_n(p)$ by  
\beas 
m_n(p)
&=&
\frac{1}{(\psi_2^n)^2} \sum_{i=0}^{p} \int_{t_{ik_n}}^{t_{(i+1)k_n}}
\Big[
2 \Delta_n^{-1/4}\alpha_{t_{ik_n}} b_{t_{i k_n}}^{[1]} G_i(r) %\mathbbm{1}_{\{r\leq t\}}
\Big]^2\>dr
\\&&
+ 
\frac{1}{(\psi_2^n)^2} \sum_{i=0}^{p} 
\int_{t_{ik_n-1}}^{t_{(i+1)k_n-1}}
[2 \Delta_n^{-1/4}\alpha_{t_{ik_n}} \omega \Delta_n^{-1/2} 
H_i(r)]^2   %{\colorr{\Bigg\}}} 
%\mathbbm{1}_{\{r\leq t\}}
\>dr. 
\eeas
By the Cauchy-Schwarz inequality, we see 
\bea\label{270815-2} 
\det\tilde \sigma\left(n,t_{(p+1)k_n}\right)
&=& 
\tilde{\sigma}_{11}(n,t_{(p+1)k_n})\sigma_{22}(t_{(p+1)k_n})-\tilde{\sigma}_{12}(n,t_{(p+1)k_n})^2
\nn\\&\geq& 
m_n(p)\>\sigma_{22}(t_{(p+1)k_n}). 
\eea
Let $u_n=\lfloor d_n/2 \rfloor k_n\Delta_n$ and let $p_n=\lfloor d_n/2\rfloor -1$. 
Set $\hat{m}_n=m_n(p_n)$. 
Define $s_n$ by 
\beas 
 s_n 
 &=&
 \frac{1}{2}\left[ \hat{m}_n\sigma_{22}(u_n)+ \psi\left(\frac{\hat{m}_n}{2 c_1}\right) \right],
\eeas
where $\psi:\bbR\to[0,1]$ in $C^\infty(\bbR)$ and satisfies $\psi(x)=1$ if $|x|\leq1/2$ and $\psi(x)=0$ if $|x|\geq1.$
Then, we observe that $s_n \geq 1/2$ if $\hat{m}_n \leq c_1$, 
and $s_n \geq 2^{-1} \hat{m}_n\sigma_{22}(u_n)$ otherwise.
\begin{en-text}
\beas 
 s_n 
 &=&
 \frac{1}{2} \det \left[ \tilde \sigma\left(n,u_n\right)+ \psi\left(\frac{m_n}{2 c_1}\right) I_2 \right], 
\eeas

Then, we observe that $s_n \geq 1/2$ if $m_n \leq c_1$, 
and $s_n \geq 2^{-1} \det \tilde \sigma\left(n,u_n\right)$ otherwise.
\end{en-text}
\begin{lemma}\label{270815-3} 
For each $q>1$, 
\beas 
\limsup_{n\to\infty}\E[s_n^{-q}]&<&\infty.
\eeas
\end{lemma}
\proof The result easily follows from (\ref{20150817-1}). 
\qed
\halflineskip

Let 
\beas 
m_n^\dagger
&=& 
\frac{4}{\psi_2^2}
\sum_{i=0}^{p_n} 
\Delta_n^{-1/2}
 \big([b_{t_{i k_n}}^{[1]} \widebar W_{t_{i k_n}}]^2+[\widebar \eps_{t_{i k_n}}]^2\big)
 \bigg( (b_{t_{i k_n}}^{[1]})^2 \psi_2^n k_n \Delta_n +\frac{\omega^2 \psi_1^n}{k_n}\bigg). 
\eeas
Then, for every $q>1$ the orthogonality between $\overline{W}_{t_{ik_n}}$ and $\overline{\ep}_{t_{ik_n}}$ 
gives   
\beas 
%\sup_{(t,p):\>t=t_{(p+1)k_n},\atop \>p=0,...,d_n-1}
\big\|\hat{m}_n-m_n^\dagger\big\|_q &=& O(\Delta_n^{1/4})
\eeas

\begin{en-text}
where 
\beas 
m_n 
&=& 
\frac{1}{\psi_2^n} \sum_{i=0}^{\lfloor d_n/2 \rfloor} 
\big\{
\Delta_n^{-1/2}[\widebar W_{t_{i k_n}}]^2+ \Delta_n^{-1/2}[\widebar \eps_{t_{i k_n}}]^2\big\}
\eeas
\end{en-text}

Let $\Pi^n=\{t_{ik_n};i=0,...,d_n-1\}$. 
\begin{lemma}\label{270816-6}
For sufficiently small positive number $c_1$, 
\beas
\sup_{t\geq1/2}\bbP\big[\det\sigma_{(M^n_t,C)}<s_n\big] 
&=&O(\Delta_n^\xi)
\eeas
where $\xi$ is arbitrary positive number. 
\end{lemma}

\begin{proof} 
With the help of (\ref{270815-2}), Lemmas \ref{270815-1} and \ref{270815-3}, we have 
\beas &&
\sup_{t\geq1/2}\bbP\big[\det\sigma_{(M^n_t,C)}<s_n\big] \leq
\sup_{t\geq1/2}\bbP\big[\det\sigma(n,t)<s_n\big] 
\\&\leq&
\sup_{t\in\Pi^n;t\geq1/2}\bbP\big[\det\sigma(n,t)<1.5s_n\big] 
\\&&
+\sup_{s,t;|s-t|\leq k_n\Delta_n}\bbP\big[|\det\sigma(n,t)-\det\sigma(n,s)|>0.5s_n\big]
\\&\leq&
\sup_{t\in\Pi^n;t\geq1/2}\bbP\big[\det\tilde{\sigma}(n,t)<2s_n\big] 
\\&&
+\sup_{t\in\Pi^n;t\geq1/2}\bbP\big[|\tilde{\sigma}(n,t)-\sigma(n,t)|>0.5s_n\big]
+O(\Delta_n^\xi)
\\&\leq&
\sup_{p\geq p_n}\bbP\big[m_n(p)\>\sigma_{22}(t_{(p+1)k_n})<2s_n\big] 
\\&&
+\sup_{t\in\Pi^n;t\geq1/2}\bbP\big[|\tilde{\sigma}(n,t)-\sigma(n,t)|>0.5s_n\big]
+O(\Delta_n^\xi)
\\&\leq&
\bbP\big[m_n(p_n)\>\sigma_{22}(u_n)<2s_n\big] 
+O(\Delta_n^\xi)
\\&\leq&
\bbP\big[\hat{m}_n\leq c_1\big] +O(\Delta_n^\xi)
\\&\leq&
\bbP\big[m_n^\dagger\leq c_1\big] +O(\Delta_n^\xi)=O(\Delta_n^\xi)
\eeas
where $\xi$ can be any positive number, 
if we chose a sufficiently small positive number $c_1$. 
\end{proof}

\subsection{Composition of $\xi_n$}\label{SecKsi}
We again consider $\psi:\bbR\to[0,1]$ in $C^\infty(\bbR)$ that satisfies 
$\psi(x)=1$ if $|x|\leq1/2$ and $\psi(x)=0$ if $|x|\geq1$. 
For $r_n=\Delta_n^{1/4}$, let 
\beas 
\xi_n'' &=& %c^*_0
r_n^{-c^*/2}(C_n-C)+2[1+4\Delta_{(M_n,C)}s_n^{-1}]^{-1}
+r_n^{c^*_1/2}C^2,
\eeas
where %$c^*_0$, 
$c^*_1>0$, and $c^*\in(2q,1)$ for $q=(1-a)/2$ for a fixed $a\in(0,1/3)$. 
Let 
\beas
\tilde{\xi}_n 
&=& 
%L^*
\int_{[0,1]^2}
\bigg(\frac{r_n^{-q/2}|C^n_t-C_t-C^n_s+C_s|}{|t-s|^{3/8}}\bigg)^8dtds,
\eeas
where %$L^*$ will be chosen as a sufficiently large constant, and 
$C^n_t=\langle M^n\rangle_t$. 
We compose $\xi_n$ as $\xi_n=\xi_n''+\tilde{\xi}_n$. 

Tracing the derivation of the stochastic expansion, we can see  
the expansion holds in $\bbD_{s,p}$ sense and 
Condition [B2]$_\ell$ holds for arbitrary $\ell\in\bbN$. 
It is easy to verify [B3] (i) by estimating 
$\bbP\big[2[1+4\Delta_{(M_n,C)}s_n^{-1}]^{-1}>2/5\big]$
with the aid of Lemma \ref{270816-6}.  
Condition [B3] (ii) is immediate by definition. 
Condition [B3] (iii) follows from Lemma \ref{270815-3}.

\subsection{Estimate of $\Phi_n^\alpha$} 
We shall verify Condition [B5] to prove the validity of the asymptotic expansion. 
We know 
\beas 
\Phi_n^\alpha(u,v) 
&=&
{\tt i}^{-|\alpha|}d^\alpha_{(u,v)}
\bbE\bigg[\int_0^1 e^n_t(u)d({\tt i}uM^n_t)\Psi(u,v)\psi_n\bigg]
\eeas
with 
$\Psi(u,v)=\exp((-u^2)+iv)C)$. 
The FGH-decomposition (\cite{YOSH13}) will be used: 
\beas 
e^n_t(u,v)\Psi(u,v) &=& \bbF^n_t(u,v)\bbG^n_t(u)\bbH^n_t(u)
\eeas
with
\beas 
\bbF^n_t(u,v) &=& \exp\left(\tti uM^n_t+\tti vC\right),
\\
\bbG^n_t(u) &=& \exp\left(-\half u^2(C-C_t)\right),
\\
\bbH^n_t(u) &=& \exp\left(\half u^2(C^n_t-C_t)\right). 
\eeas

Applying the duality twice, we obtain the representation 
\beas 
\Delta_n^{-1/4}\Phi_n^0(u,v) 
&=&
\frac{2\Delta_n^{-1/2}}{\psi_2^n}
\sum_{i=0}^{d_n-1}
\int_0^1\int_0^1dtds\>\cale^n_i(u,v)_{t,s}
\eeas
where 
\beas 
\cale^n_i(u,v)_{t,s}
&=& 
\bbE\big[\bbF^n_t(u,v)\bbG^n_t(u)\bbH^n_t(u)
\Xi^n_i(u,v)_{t,s}\big]. 
\eeas
The functional $\Xi^n_{t,s}(u,v)$ is given by 
\beas &&
\Xi^n_i(u,v)_{t,s}
\\&=&
\tti u \big(e^n_t(u)\Psi(u,v)\big)^{-1}
\\&&\times
\bigg\{
g_n(s-\tik)g_n(t-\tik)b^{[1]}_\tik D^{(1)}_s\bigg(e^n_t(u)b^{[1]}_\tik D^{(1)}_t(\Psi(u,v)\psi_n)\bigg)
\\
&+&g_n(s-\tik)\Delta_n^{-1/2}h_n(t-\tik)b^{[1]}_\tik (-\omega)
D^{(1)}_s\bigg(e^n_t(u)\Psi(u,v)D^{(2)}_t\psi_n\bigg)
\\
&+&\Delta_n^{-1/2}h_n(s-\tik)g_n(t-\tik)b^{[1]}_\tik (-\omega)
D^{(2)}_s\bigg(e^n_t(u)b^{[1]}_\tik D^{(1)}_t(\Psi(u,v)\psi_n)\bigg)
\\
&+&\Delta_n^{-1/2}h_n(s-\tik)\Delta_n^{-1/2}h_n(t-\tik)(-\omega)^2
D^{(2)}_s\bigg(e^n_t(u)\Psi(u,v)D^{(2)}_t\psi_n\bigg)
\bigg\}
\eeas
After all, $\Xi^n_i(u,v)_{t,s}$ is a polynomial of densities of $O(1)$ that are stable in $L^p$ sense. 
We note that the functions $t\mapsto \Delta_n^{-1/2}h_n(t-\tik)$ are stable 
as $n\to\infty$. % in these procedures. 

By the integration-by-parts formula applied at most 8 times, 
following the (a)-(h) procedure (pages 911-912 of \cite{YOSH13}), 
we can obtain the estimate 
\beas 
\limsup_{n\to\infty}\sup_{i=0,...,d_n-1}
\sup_{t,s\in[0,1]}
\sup_{(u,v)\in\Lambda^0_n(2,q)}
|(u,v)|^3|\cale^n_i(u,v)_{t,s}| &<& \infty. 
\eeas
Indeed, for $t<1/2$, we take advantage of the decay of $\bbG^n_t(u)$ in $u$ and 
the nondegeneracy of $C$ for $v$, i.e., (\ref{20150817-1}). 
For $t\geq1/2$, we can use nondegeneracy of $(M^n_t,C)$ for $(u,v)$, 
i.e., Lemmas \ref{270816-6} and \ref{270815-3}. 
Estimation of $\Phi_n^\alpha(u,v)$ is similarly done. 
In this way, Condition [B5] has been verified. 
%}}

\subsection{Proof of Proposition \ref{UpperSigmaLemma} and Theorem \ref{MainTh}}
%\koko Need revision of this subsection
%
Since $e_t^n(u)$ is an exponential martingale 
and $C$ is bounded on the event $\{\psi_n>0\}$, we have 
$$\Phi_n(u,v)=\E \left[\int_0^1 e_t^n(u) d(\im u M_t^n) \Psi(u,v) \psi_n \right],$$
where
$$\Psi(u,v)=\exp{\left((-u^2/2+\im v)C \right)}. $$
We decompose $\Delta_n^{-1/4} \Phi_n(u,v)$ as 
$$ \Delta_n^{-1/4} \Phi_n(u,v)=\check {U}_n(u,v)+\hat {U}_n(u,v),$$
where
\begin{align*}
\check {U}_n(u,v)&= \Delta_n^{-\frac{1}{4}} \sum_{i=0}^{d_n-1} 
\E \left[\Psi(u,v) \int_{t_{i k_n}}^{t_{(i+1) k_n}}
 e_{t_{i k_n-1}}^n(u) d(\im u M_t^n) \psi_n \right], \\
\hat {U}_n(u,v)&= \Delta_n^{-\frac{1}{4}} \sum_{i=0}^{d_n-1} 
\E \left[\Psi(u,v) \int_{t_{i k_n}}^{t_{(i+1) k_n}}
 (e^n_t(u)- e_{t_{i k_n-1}}^n(u)) d(\im u M_t^n) \psi_n \right].
\end{align*}
We will prove that $\check {U}_n(u,v)$ is the main term and $\hat {U}_n(u,v)$ is negligible. 

Let us first look at the term $\check {U}_n(u,v)$. We observe that 
\beas 
\check {U}_n(u,v)
&=&
\frac{2\Delta_n^{-1/2}}{\psi_2^n}
\sum_{i=0}^{d_n-1}
\int_0^1\int_0^1dtds\>\check{\cale}^n_i(u,v)_{t,s}
\eeas
where $\check{\cale}^n_i(u,v)_{t,s}$ has a representation similar to $\cale^n_i(u,v)_{t,s}$ 
as 
\beas 
\check{\cale}^n_i(u,v)_{t,s} 
&=& 
\E \big[e^n_{\tik-1}(u)\Psi(u,v)\check{\Xi}^n_i(u,v)_{t,s}\big]
\eeas
with 
\begin{align*}
&\check{\Xi}^n_i(u,v)_{t,s}=
\tti u \big(e^n_{\tik-1}(u)\Psi(u,v)\big)^{-1}
\\&\times
\bigg\{
g_n(s-\tik)g_n(t-\tik)b^{[1]}_\tik e^n_{\tik-1}(u)b^{[1]}_\tik 
D^{(1)}_s\bigg(D^{(1)}_t(\Psi(u,v)\psi_n)\bigg)
\\&
+g_n(s-\tik)\Delta_n^{-1/2}h_n(t-\tik)b^{[1]}_\tik (-\omega)
e^n_{\tik-1}(u)D^{(1)}_s\bigg(\Psi(u,v)D^{(2)}_t\psi_n\bigg)
\\&
+\Delta_n^{-1/2}h_n(s-\tik)g_n(t-\tik)b^{[1]}_\tik (-\omega)
e^n_{\tik-1}(u)b^{[1]}_\tik D^{(1)}_t\bigg(\Psi(u,v)D^{(2)}_s\psi_n\bigg)
\\&
+\Delta_n^{-1/2}h_n(s-\tik)\Delta_n^{-1/2}h_n(t-\tik)(-\omega)^2
e^n_{\tik-1}(u)\Psi(u,v)D^{(2)}_sD^{(2)}_t\psi_n
\bigg\}. 
\end{align*}
For derivation of the above equality, the bounds of the supports of $g_n$ and $h_n$ 
were  used. 
The terms stemming from the last three terms in $\{...\}$ are negligible 
since $\P[|\xi_n|>1]=O(\Delta_n^L)$ for arbitrary $L>0$. 
By the same reason applied to the first term there, we have 
\beas 
\check {U}_n(u,v)
&=&
\frac{2\Delta_n^{-1/2}}{\psi_2^n}
\sum_{i=0}^{d_n-1}
\int_0^1\int_0^1dtds\>\dot{\cale}^n_i(u,v)_{t,s}+o(1)
\eeas
where 
\beas 
\dot{\cale}^n_i(u,v)_{t,s}
&=&
\E \big[e^n_{\tik-1}(u)\Psi(u,v)\dot{\Xi}^n_i(u,v)_{t,s}\big]
\eeas
and
\beas &&
\dot{\Xi}^n_i(u,v)_{t,s}
\\&=&
\tti u \big(e^n_{\tik-1}(u)\Psi(u,v)\big)^{-1}
\\&&\times
g_n(s-\tik)g_n(t-\tik)(b^{[1]}_\tik)^2 e^n_{\tik-1}(u)%b^{[1]}_\tik 
\psi_n
D^{(1)}_s D^{(1)}_t\Psi(u,v)
\\&=&
\tti u \>g_n(s-\tik)g_n(t-\tik)(b^{[1]}_\tik)^2\psi_n
\big\{l^2(D_s^{(1)}C)(D_t^{(1)}C)+lD_s^{(1)}D_t^{(1)}C\big\}
\eeas
where $l=-\frac{u^2}{2}+ \im v$. 
Therefore 
\begin{align*}
 &\check {U}_n(u,v)
=
\frac{2\Delta_n^{-1/2}}{\psi_2^n}\sum_{i=0}^{d_n-1}
\int\int dtds \>
\E \bigg[e^n_\tik(u)\Psi(u,v) \tti u
\>g_n(s-\tik) g_n(t-\tik)\\
&\times  (b^{[1]}_\tik)^2\psi_n
\big\{l^2(D_s^{(1)}C)(D_t^{(1)}C)+lD_s^{(1)}D_t^{(1)}C\big\}\bigg] \to
\\ 
&\frac{\theta\psi_3^2}{\psi_2}
\int_0^1 \E\bigg[
\tti u \>\exp \big(\tti u M_t+C\frac{u^2}{2}\big)\Psi(u,v)   (b^{[1]}_t)^2
\big\{l^2(D_t^{(1)}C)^2+lD_t^{(1)}D_t^{(1)}C\big\}\bigg]\>dt
\\&=
\E\bigg[\Psi(u,v) \>
\frac{\theta\psi_3^2}{\psi_2}\int_0^1 \tti u \> (b^{[1]}_t)^2
\big\{l^2(D_t^{(1)}C)^2+lD_t^{(1)}D_t^{(1)}C\big\}\>dt\bigg]
\\&=
\E\bigg[\Psi(u,v) \>
\frac{\theta\psi_3^2}{\psi_2} \tti u \>\big\{4\theta^2
l^2\calc_2+2\theta l (\calc_3+\calc_4)\big\}\bigg]
\end{align*}
where $D_t^{(1)}D_t^{(1)}C=\lim_{s\uparrow t}D_s^{(1)}D_t^{(1)}C$ and 
\beas
\mathcal{C}_2
&=& 
\int_0^1 (b^{[1]})^2(X_t) \left( \int_t^1 c'(X_r) D^{(1)}_t X_r dr  \right)^2 dt, \\
\mathcal{C}_3
&=&  
\int_0^1 (b^{[1]})^2(X_t) \left( \int_t^1 c''(X_r) (D^{(1)}_t X_r)^2 dr  \right) dt, \\
\mathcal{C}_4
&=&  
\int_0^1 (b^{[1]})^2(X_t) \left( \int_t^1 c'(X_r) D^{(1)}_t D^{(1)}_t X_r dr  \right) dt
\eeas
for $c(x)=\left[(b^{[1]})^2(x)+ \frac{\omega^2 \psi_1}{\theta^2 \psi_2} \right]^2$. 
%}}
%
\begin{en-text}
\beas
\check {U}_n(u,v)
 & \to &
 %\frac{\theta \psi_3 }{\psi_2} \int_0^1  \E \left[ \im u (b_{t}^{[1]})^2 e_t(u) D_{t,t} \Psi(u,v)   \right] dt 
% \\
\frac{\theta \psi_3 }{\psi_2} \E \left[ \im u  (4 \theta^2 l^2 \mathcal{C}_2+ 2\theta l (\mathcal{C}_3+\mathcal{C}_4)) \Psi(u,v)   \right] 
 \\&&
\eeas
\end{en-text}
\begin{en-text}
Let us first look at the term $\check {U}_n(u,v).$ We observe that 
\begin{align*} \int_{t_{i k_n-1}}^{ t_{(i+1) k_n-1}} d(M_t^n)=&
\frac{ \Delta_n^{-1/4}}{\psi_2^n} \Big[ (b_{t_{i k_n}}^{[1]})^2 \delta^2 (g_n(\cdot-t_{i k_n})^{\otimes 2}) \\
&+\omega^2 \Delta_n^{-1} \delta^2 (h_n(\cdot-t_{i k_n})^{\otimes 2})+ 2 b_{t_{i k_n}}^{[1]} \widebar W_{t_{i k_n}} \widebar \eps_{t_{i k_n}}  \Big] ,\end{align*}
where $\delta^2$
denotes the double stochastic integral. We note that $\widebar \eps_{t_{i k_n}}$ is independent of $\mathcal{F}_t$ and has mean 0. Applying the integration by parts formula, we obtain
\begin{align*}
\check {U}_n(u,v)=& \frac{\im u \Delta_n^{-1/2}}{\psi_2^n} \sum_{i=0}^{d_n-1} \Bigg\{ \E \left[ (b_{t_{i k_n}}^{[1]})^2 \delta^2(g_n(\cdot-t_{i k_n})^{\otimes 2}) \Psi(u,v)  e_{t_{i k_n-1}}^n(u) \psi_n \right] \\
 &+ \E \left[ \omega^2 \Delta_n^{-1} \delta^2(h_n(\cdot-t_{i k_n})^{\otimes 2}) \Psi(u,v)  e_{t_{i k_n-1}}^n(u) \psi_n \right]
 \Bigg \} \\
 & \to \frac{\theta \psi_3 }{\psi_2} \int_0^1  \E \left[ \im u (b_{t}^{[1]})^2 e_t(u) D_{t,t} \Psi(u,v)   \right] dt \\
 &= \frac{\theta \psi_3 }{\psi_2} \E \left[ \im u  (4 \theta^2 l^2 \mathcal{C}_2+ 2\theta l (\mathcal{C}_3+\mathcal{C}_4)) \Psi(u,v)   \right]
\end{align*}
\end{en-text}
Convergence 
$\hat{U}_n(u,v)\to0$ is easy to show. 
Moreover, it is possible to specify the limit $\Phi^\alpha$ in a similar way 
to verify (\ref{characterize.upperbar.general}) for $F=C$. Hence, we obtain 
$$\overline \sigma (\im u, \im v)=\frac{\theta \psi_3^{2}
 }{\psi_2} \im u[4 \theta^2 \mathcal{C}_2 l^2+ 2\theta (\mathcal{C}_3+\mathcal{C}_4)l ].
$$
Thus, Proposition \ref{UpperSigmaLemma} and hence 
Condition [B4]$_{\ell,2,1}$ has been verified, which concludes the proof of Theorem \ref{MainTh}. 
\qed

\bibliographystyle{plain} %{alpha}
\bibliography{Bezirgen_01}

\end{document}